\documentclass[12pt]{amsart}
\usepackage[utf8x]{inputenc}
\usepackage{ucs}
\usepackage{amsmath}
\usepackage{amsfonts}
\usepackage{amssymb}
\author{Pierre Tarrago}
\title{Sawtooth models and asymptotic independence in large compositions}
\usepackage{amsbsy}
\usepackage{amscd} 
\usepackage{hyperref}
\usepackage{amsgen} 
\usepackage{amsopn} 
\usepackage{amstext}
\usepackage{amsthm} 
\usepackage{amsxtra}
\usepackage{rotating}
\usepackage{tikz}
\usepackage{genyoungtabtikz}
\newtheorem{lem}{Lemme}
\newtheorem{prop}{Proposition}
\newtheorem{thm}{Theorem}
\newtheorem{defn}{Definition}
\newtheorem{cor}{Corollary}
\newtheorem{rmq}{Remark}

\newtheorem{nota}{Notation}
\newtheorem*{lem*}{Lemme}
\newtheorem*{prop*}{Proposition}
\newtheorem*{thm*}{Theorem}
\newtheorem*{defn*}{Définition}
\newtheorem*{cor*}{Corollary}

\begin{document}
\maketitle
\begin{abstract}
In this paper we improve the probabilistic approach of Ehrenborg, Levin and Readdy in \cite{ehrenborg2002probabilistic} by introducing a simpler but more general probabilistic model. As consequence we get some new estimates on the behavior of a uniform random permutation $\sigma$ having a fixed descent set. In particular we find a positive answer to the Conjecture 4 of \cite{bender2004asymptotics} and we show that independently of the shape of the descent set, $\sigma(i)$ and $\sigma(j)$ are almost independent when $i-j$ becomes large. 
\end{abstract}
\section{Introduction}
A descent of a permutation $\sigma$ of $n\in\mathbb{N}^{*}$ is an integer $i$ such that $\sigma(i)>\sigma(i+1)$. For each permutation $\sigma$, the corresponding descent set $D(\sigma)$ is the set of all the descents of $\sigma$. Since descents can be located everywhere except on $n$, a descent set is just a subset of $\lbrace 1,\dots,n-1\rbrace$, and for the moment we call a composition of $n$ the data of $n$ and a subset of $\lbrace 1,\dots,n-1\rbrace$.  We can pictiorally reformulate this by drawing a composition $D$ as a skew Young diagram $\lambda_{D}$ of $n$ cells $1,\dots,n$ with the following rule : cells $i$ and $i+1$ are neighbors and the cell $i+1$ is right to $i$ if $i\not\in D$, below $i$ otherwise. Therefore the descent set of a permutation $\sigma$ is $D$ if and only if inserting $\sigma(i)$ in each cell $i$ of $\lambda_{D}$ results in a standard skew-Young tableau. For example the composition $D=\lbrace 10, (3,5,9)\rbrace$ matches the following skew Young diagram:
\begin{figure}[!h]
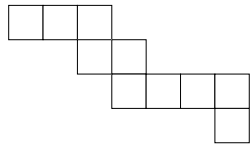


\gyoung(;;;,::;;,:::;;;;,::::::;)
\caption{\label{fig1}Skew Young diagram $\lambda_{D}$ associated to the composition $D=\lbrace 10,(3,5,9)\rbrace$}
\end{figure}\\
And the permutation $\sigma=(3,5,8,4,7,1,6,9,10,2)$ has the descent set $D$ since the associated filling of $\lambda_{D}$ results in a skew Young tableau as shown in figure \ref{fig2}.
\begin{figure}[!h]
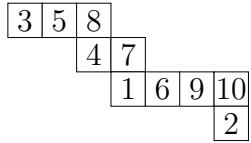

\gyoung(;3;5;8,::;4;7,:::;1;6;9;<10>,::::::;2)
\caption{\label{fig2}Standard filling of the composition $(3,2,4,1)$}
\end{figure}\\

Conversely for each composition $D$ of $n$, the problem is to count how many permutations of $[1,n]$ have exactly $D$ as descent set; it is equivalent to count the number of standard fillings of the associated skew Young tableau $\lambda_{D}$. This latter number, $\beta(D)$, is called the descent statistic of $D$ and has been intensively studied in the last decades (see Viennot \cite{viennot1979permutations} and \cite{viennot1981equidistribution} , Niven \cite{niven1968combinatorial}, de Bruijn \cite{de1970permutations} , ...): the main questions were on one hand to find the compositions of $n$ having a maximum descent statistic, and on the other hand to find exact or asymptotic formulae for descent statistic of compositions of given shape and large size. For example, Niven and de Bruijn proved in \cite{niven1968combinatorial} and \cite{de1970permutations} that the two compositions of $n$ maximizing the descent statistic are $D_{1}(n)=\lbrace 1,3,5,\dots\rbrace\cap[1,n]$ and $D_{2}(n)=\lbrace 2,4,6,\dots,\rbrace\cap[1,n]$, whose associated permutations are called alternating permutations; Désiré André gave long before them in \cite{andre1881permutations} an asymptotic formula for the number of alternating permutations, showing that $\beta(D_{1})(n)\sim 2(2/\pi)^{n}n!$ as $n$ goes to infinity.\\
To be able to evaluate the descent statistic of a broad class of compositions, Ehrenborg, Levin and Readdy formalized in \cite{ehrenborg2002probabilistic} a probabilistic approach to the counting problem, by relating each permutation of $[1,n]$ with a particular simplex of $[0,1]^{n}$. Since the cube $[0,1]^{n}$ with the Lebesgue measure can be seen as a probability space, it is possible to use probabilistic tools to get interesting results on descent statistics : Ehrenborg obtained in \cite{ehrenborg2002asymptotics} asymptotic descent statistics for the so-called nearly periodic permutations, which consists essentially in permutations having the same descent pattern repeated several times and with some local perturbations. Once again the asympotic formula has the shape $K\lambda^{n}n!$, with $K$ and $\lambda$ some constants depending on the situation. Using this approach together with functional analysis tools, Bender, Helton and Richmond extended in \cite{bender2004asymptotics} the latter result to a broader class of descent sets, and found asymptotic formulae of the same shape as before.  The factorial term of the asymptotic formula is easy to get, since it comes from the cardinality of the set $\mathfrak{S}_{n}$ of permutations of $n$ elements. However the power term is harder to understand. The main point of the article \cite{bender2004asymptotics} is that the authors identified in this class of descent sets the phenomenon that makes the power term $\lambda^{n}$ appear: namely if we consider a large uniformly random permutation with a fixed descent set, the value of $\sigma(1)$ and $\sigma(n)$ are nearly independent, which causes a factorization in the asymptotic counting. The natural question is thus to know which compositions induce this phenomenon, and it was conjectured in \cite{bender2004asymptotics} that every composition have this property as they become large.\\
In the present article we construct a family of particular statistic models, called sawtooth models, that greatly simplifies the probabilistic approach of Ehrenborg, Readdy and Levin. These models are more general than the ones we need in the combinatoric of descent sets, but the properties we will use thereafter appear more clearly in this broader case; thus we first study these models in their full generality, before deducing some specific results on descent sets. As a main consequence we derive an affirmative answer to the Conjecture 4 on asymptotic independence from Bender, Helton and Richmond (\cite{bender2004asymptotics}) and we are able to conclude by the following intuitive result on compositions :\\
\textit{In the random filling of a composition, the content of two distant cells are almost independent.}\\
In a forthcoming paper we will use the results of this article to study an analog of the Young lattice that was introduced by Gnedin and Olshanski in \cite{gnedin2006coherent}.
\section{Preliminaries and results}
\subsection{Compositions}
This paragraph gives definitions and notations concerning compositions.
\begin{defn}\label{composition}
Let $n\in\mathbb{N}$. A composition $\lambda$ of $n$ is a sequence of positive integers $(\lambda_{1},\dots,\lambda_{r})$ such that $\sum \lambda_{j}=n$.
\end{defn}
A unique ribbon Young diagram with $n$ cells is associated to each composition: each row $j$ has $\lambda_{j}$ cells, and the first cell of the row $j+1$ is just below the last cell of the row $j$. For example the composition of $10$, $(3,2,4,1)$ is represented as in figure \ref{fig1}.
This picture shows directly the link between Definition \ref{composition} and the definition we stated in the introduction :  a composition $\lambda=(\lambda_{1},\dots,\lambda_{r})$ of $n$ yields a subset $D_{\lambda}$ of  $\lbrace 1,\dots,n-1\rbrace$, namely the subset $\lbrace \lambda_{1},\lambda_{1}+\lambda_{2},\dots,\lambda_{1}+\dots+\lambda_{r-1}\rbrace$. The latter correspondence is clearly bijective.\\
The size $\vert \lambda\vert$ of a composition is the sum of the $\lambda_{j}$. When nothing is specified, $\lambda$ will always be assumed to have the size $n$, and $n$ will always denote the size of the composition $\lambda$. \\
A standard filling of a composition $\lambda$ of size $n$ is a standard filling of the associated ribbon Young diagram: this is an assignement of a number between $1$ and $n$ for each cell of the composition, such that every cells have different entries, and the entries are increasing to the right along the rows and decreasing to the bottom along the columns. An example for the composition of figure \ref{fig1} is shown in figure \ref{fig2}.\\
In particular, reading the tableau from left to right and from top to bottom gives for each standard filling a permutation $\sigma$; moreover the descent set of such a  $\sigma$, namely the set of indices $i$ such that $\sigma(i+1)<\sigma(i)$, is exactly the set 
$$D_{\lambda}=\lbrace \lambda_{1},\lambda_{1}+\lambda_{2},\dots,\sum_{1}^{r-1}\lambda_{i}\rbrace.$$
There is a bijection between the standard fillings of $\lambda$ and the permutations of $\vert \lambda\vert$ with descent set $D_{\lambda}$. For example the filling in figure \ref{fig2} yields the permutation $(3,5,8,4,7,1,6,9,10,2)$.\\
\subsection{Result on asymptotic independence}
We present here the main results that are proven in the present paper.
\begin{nota}
Let $\lambda$ be a composition. Let $\Sigma_{\lambda}$ denote the set of all permutations with descent set $D_{\lambda}$. With the uniform counting measure $\mathbb{P}_{\lambda}$ it becomes a probability space, and $\sigma_{\lambda}$ denotes the random permutation coming from this probability space.
As usual $\vert\Sigma_{\lambda}\vert$ is the cardinal of the set $\Sigma_{\lambda}$. 
\end{nota}
$\vert\Sigma_{\lambda}\vert$ is thus the descent statistic associated to the composition $\lambda$.\\
Denote for each random variable $X$ by $\mu(X)$ its law and by $d_{X}$ its density, and write $\mu\otimes \nu$ the independent product of two laws. The goal of the paper is to prove that distant cells in a composition have independent entries, namely:
\begin{thm}\label{goal2}
Let $\epsilon,r\in\mathbb{N}$. Then there exists $n\geq 0$ such that if $\lambda$ is a composition of $N$ and $0<i_{1}<\dots <i_{r}\leq N$ are indices with $i_{j+1}-i_{j}\geq n$,
$$d_{\pi}\left(\mu(\frac{\sigma_{\lambda}(i_{1})}{N},\dots,\frac{\sigma_{\lambda}(i_{r})}{N}),\mu(\frac{\sigma(i_{1})}{N})\otimes\dots\otimes \mu(\frac{\sigma(i_{r})}{N})\right)\leq \epsilon,$$
with $d_{\pi}$ denoting the Levy-Prokhorov metric on the set of measures of $[0,1]^{r}$.
\end{thm}
If the first and last runs of the composition remain bounded, the latter can be improved for the density of the first and last particle. This is the content of the Conjecture $4$ of \cite{bender2004asymptotics} that is proven in this paper and reformulated here in term of permutation :
\begin{thm}\label{goal1}
Let $\epsilon>0$, $A\geq 0$. There exists $n\geq 0$ such that for any composition $\lambda$ of size larger than $n$ with first and last run bounded by $A$,
\begin{equation}
\Vert d_{(\frac{\sigma_{\lambda}(1)}{n},\frac{\sigma_{\lambda}(n)}{n})}-d_{\frac{\sigma_{\lambda}(1)}{n}}d_{\frac{\sigma_{\lambda}(n)}{n}}\Vert_{\infty}<\epsilon.
\end{equation}
\end{thm}
\subsection{Runs of a composition}
Let $\lambda$ be a composition. We number the cells as we read them, from left to right and from top to bottom \label{coor}. The cells are identified with integers from $1$ to $n$ through this numbering. For example in the standard filling of figure (\ref{fig2}), the number $7$ is in the cell $5$.\\
We call run any set consisting in all the cells of a given column or row. The set of runs is ordered with the lexicographical order. In the same example as before the runs are 
$$s_{1}=(1,2,3),s_{2}=(3,4),s_{3}=(4,5),s_{4}=(5,6),s_{5}=(6,7,8,9),s_{6}=(9,10),$$
where we put in the parenthesis the cells of each run.\\
Note that inside each run the cells are ordered by the natural order on integers. We call extreme cell a cell that is an extremum in a run with respect to this order, and denote by $\mathcal{E}_{\lambda}$ the set of extreme cells of $\lambda$. Apart from the first and last cells of the composition, every extreme belong to two consecutive runs. Let $P_{\lambda}$ be the set of extreme cells followed by a column, or preceeded by a row and $V_{\lambda}$ the set of extreme cells followed by a row or preceeded by a column. The elements of $P_{\lambda}$ are called peaks and the one of $V_{\lambda}$ valleys. The sets $V_{\lambda}$ and $P_{\lambda}$ are also ordered with the natural order: 
$$P_{\lambda}=\lbrace x^{+}_{1}<\dots<x^{+}_{r}\rbrace,V_{\lambda}=\lbrace x^{-}_{1}<\dots<x^{-}_{t}\rbrace,$$
with $r-1\leq t\leq r+1$.\\
The first and last cells are always extreme points. A composition is said being of type ++ (resp. +-,-+,--) if the first cell is a peak and the last cell is a peak (resp peak-valley,valley-peak, valley-valley).\\
Finally let $l(s)$, the length  of a run $s$, be the cardinal of $s$, and $L(\lambda)$, the amplitude of $\lambda$,  be the supremum of all lengths.

\subsection{The coupling method}
In this paragraph we introduce a probabilistic tool called the coupling method, and set the relative notations for the sequel. We refer to \cite{lindvall2002lectures} for a review on the subject. We will present the notions in the framework of random variables but we could have done the same with probability laws as well.
\begin{defn}
Let $(E,\mathcal{E})$ be a probability space and $X,Y$ two random variables on $E$. A coupling of $(X,Y)$ is a random variable $(Z_{1},Z_{2})$ on $(E\times E,\mathcal{E}\otimes,\mathcal{E})$ such that 
$$Z_{1}\sim_{law} X,Z_{2}\sim_{law} Y.$$
\end{defn}
Such a coupling always exists : it suffices to consider two independent random variables $Z_{1}$ and $Z_{2}$ with respective law $\mu_{X}$ and $\mu_{Y}$. However a coupling is often useful precisely when the resulting random variables $Z_{1}$ and $Z_{2}$ are far from being independent. In particlular in this article we are mainly interested in the case where $Z_{1}$ and $Z_{2}$ respect a certain order on the set $E$. From now on $E$ is a Polish space considered with its borelian $\sigma-$algebra $\mathcal{E}$, and $\prec$ a partial order on $E$ such that the graph $\mathcal{G}=\lbrace (x,y),x\prec y\rbrace$ is $\mathcal{E}-$measurable.
\begin{defn}
Let $X,Y$ be two random variables on $E$. $Y$ stochastically dominates $X$ (denoted $Y\succeq X$) if and only if
$$\mathbb{P}(X\in A)\leq\mathbb{P}(Y\in A)$$
for any Borel set $A$ such that 
$$x\in A\Rightarrow \lbrace y\in E, x\prec y\rbrace\subset A.$$
\end{defn}
For example if $E=\mathbb{R}$ with the canonical order $\leq$ and $\sigma-$algebra $\mathcal{B}(\mathbb{R})$, then $Y$ stochastically dominates $X$ if and only if for all $x\in \mathbb{R}$,
$$\mathbb{P}(X\in [x,+\infty[)\leq \mathbb{P}(Y\in [x,+\infty[)$$
or equivalently, if we denote by $F_{X}(t)$ and $F_{Y}(t)$ their respective cumulative distribution function:
$$F_{Y}(t)\leq F_{X}(t)\text{ for all }t\in\mathbb{R}.$$\\
There are several ways to characterize the stochastic dominance:
\begin{prop}\label{definitionStochastDom}
The three following statements are equivalent :
\begin{itemize}
\item $Y$ stochastically dominates $X$
\item there exists a coupling $(Z_{1},Z_{2})$ of $X,Y$ such that $Z_{1}\prec Z_{2}$ almost surely.
\item for any positive measurable bounded function $f$ that is non-decreasing with respect to $\prec$,
$$\mathbb{E}(f(X))\leq \mathbb{E}(f(Y))$$
\end{itemize}
\end{prop}
The proof is straightforward and can be found in \cite{lindvall2002lectures}. This yields the following intuitive Lemma :
\begin{lem}\label{stochdom}
Let $(X_{1},X_{2}),(Y_{1},Y_{2})$ be two couples of independent random variables on $E\times E$ such that $X_{1}\preceq Y_{1}$ and $Y_{2}\preceq X_{2}$. Then 
$$\mathbb{P}(X_{1}\prec X_{2})\geq \mathbb{P}(Y_{1}\prec Y_{2}).$$
\end{lem}
\begin{proof}
Let $\ll$ be the partial order on $E\times E$ defined by
$$(x,y)\ll(x',y')\leftrightarrow x\prec x' \text{ and } y'\prec y.$$
If $Y_{1}\succeq X_{1}$ and $X_{2}\succeq Y_{2}$, there exists a coupling $(\hat{X}_{1},\hat{Y}_{1})$ (resp. $(\hat{X}_{2},\hat{Y}_{2})$) of $X_{1},Y_{1}$ (resp. $X_{2},Y_{2}$) such that almost surely $\hat{X}_{1}\prec \hat{Y}_{1}$ (resp $\hat{X}_{2}\succ \hat{Y}_{2}$). These two couplings can be chosen independent. Since $(X_{1},Y_{1})$ and $(X_{2},Y_{2})$ are also independent, this implies that $(\hat{X}_{1}\otimes\hat{X}_{2},\hat{Y}_{1}\otimes\hat{Y}_{2})$ is a coupling of $((X_{1},X_{2}),(Y_{1},Y_{2}))$ with almost surely
$$(\hat{X}_{1},\hat{X}_{2})\ll (\hat{Y}_{1},\hat{Y}_{2}).$$
But if 
$\hat{Y}_{1}\prec \hat{Y}_{2}$, then $\hat{X_{1}}\prec \hat{Y}_{1}\prec \hat{Y}_{2}\prec \hat{X}_{1}$
and thus 
$$
\mathbb{P}(Y_{1}\prec Y_{2})=\mathbb{P}(\hat{Y}_{1}\prec \hat{Y}_{2})
\leq  \mathbb{P}(\hat{X}_{1}\prec \hat{X}_{2})=\mathbb{P}(X_{1}\prec X_{2}).$$
\end{proof}
These results will be concretly applied on $\mathbb{R}^{n}, n\geq 1$, and thus we need to define a family of partial order on those sets.
\begin{defn}
Let $n\geq 1$. The partial order $\leq$ on $\mathbb{R}^{n}$ is the natural order on $\mathbb{R}$ for $n=1$, and for $n\geq 2$ if $(x_{i})_{1\leq i\leq n},(y_{i})_{1\leq i\leq n}\in \mathbb{R}^{n}$,
$$(x_{i})_{1\leq i\leq n}\leq (y_{i})_{1\leq i\leq n}\Leftrightarrow \forall i\in[1; n],  x_{i}\leq y_{i}.$$
For any word of length $n$ in $\lbrace 1,0\rbrace$, the modified partial order $\leq_{\epsilon}$ is defined as 
$$(x_{i})_{1\leq i\leq n}\leq (y_{i})_{1\leq i\leq n}\Leftrightarrow \forall i\in[1;n], (-1)^{\epsilon_{i}}x_{i}\leq (-1)^{\epsilon_{i}}y_{i}.$$
\end{defn}
The easiest way to check the stochastical dominance is to look at the cumulative distribution function. The proof of the following Lemma is a direct application of Proposition \ref{definitionStochastDom}.
\begin{lem}\label{stodomrn}
Let $(X_{i})_{1\leq i\leq n}$ and $(Y_{i})_{1\leq i\leq n}$ be two random variables of $(\mathbb{R}^{n},\leq_{\epsilon})$. Then $(Y_{i})_{1\leq i\leq n}$ stochastically dominates $(X_{i})_{1\leq i\leq n}$ if and only if for all $(t_{i})_{1\leq i\leq n}\in\mathbb{R}^{n}$,
$$F_{(X_{i})}(t_{1},\dots,t_{n})\geq_{\epsilon} F_{(Y_{i})}(t_{1},\dots,t_{n}).$$
\end{lem}
The stochastic dominance in the case $(\mathbb{R}^{n},\leq_{\epsilon})$ is denoted as $(X_{1},\dots,X_{n})\preceq_{\epsilon}(Y_{1},\dots, Y_{n})$. A consequence of the previous result is that if $(Y_{1},\dots, Y_{n})$ stochastically dominates $(X_{1},\dots,X_{n})$, then for all subsets $I=(i_{1},\dots,i_{r})$ of $\lbrace 1,\dots,n\rbrace$, $(Y_{i_{1}},\dots,Y_{i_{r}})$ also stochastically dominates $(X_{i_{1}},\dots,X_{i_{r}})$.  \\
Applying Lemma \ref{stodomrn} to the case $n=2$ yields the following Lemma:
\begin{lem}\label{stochasticDominationCondition}
Let $(U_{1},V_{1}),(U_{2},V_{2})$ be two random variables on $[0,1]$ such that $U_{2}$ and $V_{2}$ are independent. Suppose that for all $0\leq t\leq 1$, 
$$F_{V_{1}}(t)\leq F_{V_{2}}(t)$$
and for all $v\in[0,1]$,
$$F_{U_{1}\vert V_{1}=v}(t)\leq F_{U_{2}}(t).$$
There existe a coupling $((Z_{1},\tilde{Z}_{1}),(Z_{2},\tilde{Z}_{2}))$ of $(U_{1},V_{1})$ and $(U_{2},V_{2})$ such that almost surely 
$$(Z_{1},\tilde{Z}_{1})\geq (Z_{2},\tilde{Z}_{2}) .$$
\end{lem}
\section{Sawtooth model}
\subsection{Definition of the model}

In this section we introduce a statistical model of particles in a tube, which is a generalization of the probabilistic approach of Ehrenborg, Levin and Readdy in \cite{ehrenborg2002probabilistic}. The model consists in a sequence of particles, each of them moving vertically in an horizontal two-dimensional tube. Each particle has a repulsive action on the two neighbouring particles, and moreover the set of particles splits into two groups: the upper particles and the lower particles. The upper particles are always above the lower ones. The model is depicted in Figure \ref{fig3}.
\begin{figure}[!h]

\begin{tikzpicture}[scale=0.5][center]

\draw (0,0) -- (12,0);
\draw (0,6)-- (12,6);
\draw [dashed] (1,0) to (1,6);
\draw [dashed] (5,0) to (5,6);
\draw [dashed] (9,0) to (9,6);
\draw [dashed] (3,6) to (3,0);
\draw [dashed] (7,6) to (7,0);
\draw [dashed] (11,6) to (11,0);
\node [draw] [circle,fill=white](q1) at (1,2) {$q_{1}$};
\node [draw] [circle,fill=white](q2) at (5,1) {$q_{2}$};
\node [draw] [circle,fill=white](q3) at (9,3) {$q_{3}$};
\node [draw] [circle,fill=white](p1) at (3,4) {$p_{1}$};
\node [draw] [circle,fill=white](p2) at (7,5) {$p_{2}$};
\node [draw] [circle,fill=white](p3) at (11,5) {$p_{3}$};

\draw [<->,thick](q1.north east) to (p1.south west);
\draw [<->,thick](p1.south east) to (q2.north west);
\draw [<->,thick](q2.north east) to (p2.south west);
\draw [<->,thick](p2.south east) to (q3.north west);
\draw [<->,thick](q3.north east) to (p3.south west);

\end{tikzpicture}
\caption{\label{fig3}Repulsive particles in a tube}

\end{figure}
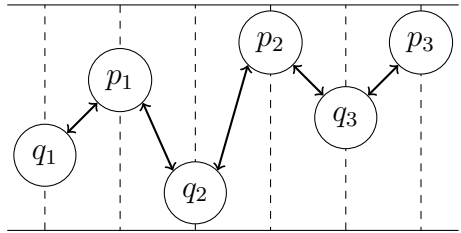\\
Such a system is called a Sawtooth model in the sequel.\\
\begin{rmq} \label{typeModel} 
If there are $n$ upper-particles, there must be $m$ lower particles with $m\in\lbrace n-1,n,n+1\rbrace$, depending on what is the type of the first and the last particles. We define therefore the type $\epsilon(\mathcal{S})$ of the model $\mathcal{S}$ as the word $\epsilon_{I}\epsilon_{F}$, with $\epsilon_{I}=+$ (resp. $\epsilon_{F}=+$) if the first (resp. last) particle is an upper one, and $\epsilon_{I}=-$ (resp. $\epsilon_{F}=-$) otherwise.
\end{rmq}
Unless specified otherwise, the first particle is a lower particle (as in the picture). The particles are ordered from the left, and following this order the upper particles are written $\lbrace p_{1}<p_{2}<\dots<p_{n}\rbrace$ and the lower particles $\lbrace q_{1}<\dots<q_{m}\rbrace$. Since the nature of our results won't depend of the type of the model, we will also assume that there are $n+1$ lower particles, yielding that the last particle is a lower one too.\\
Let $x_{i}$ be the position of $q_{i}$, $y_{i}$ the position of $p_{i}$ and denote by $\xi_{i}(x_{i},y_{i})$ (resp. $\rho_{i}(y_{i},x_{i+1})$) the potential of the repulsive force between $q_{i}$ and $p_{i}$ (resp. $p_{i}$ and $q_{i+1}$).  The probability to get a configuration $\lbrace x_{i},y_{i}\rbrace$ at the Gibbs equilibrium with a temperature $T$ is :
\begin{equation}\label{gibbs}\tag{$\ast$}
d\mathbb{P}_{Gibbs}(\lbrace x_{i},y_{i}\rbrace)=\frac{1}{\mathcal{Z}}\exp(-\frac{\sum(\xi_{i}(x_{i},y_{i})+\rho_{i}(y_{i},x_{i+1}))}{k_{B}T}).
\end{equation}
From now on we assume that the potentials only depend on the relative positions of the particles, namely $\xi_{i}(y_{i},x_{i+1})=\tilde{f}_{i}(\vert y_{i}-x_{i+1}\vert)$ and $\rho_{i}(x_{i},y_{i})=\tilde{g}_{i}(\vert x_{i+1}-y_{i}\vert)$ for some functions $\tilde{f}_{i},\tilde{g}_{i}$. Since the forces are repulsive, $\tilde{f}_{i}$ and $\tilde{g}_{i}$ must be decreasing. Moreover by a rescaling we can assume that $x_{i},y_{i}\in [0,1]$.\\
Aiming the results we stated on compositions, we should answer these questions :
\textit{
\begin{enumerate}
\item As the number of particles goes to infinity, is there some independence between $X_{1}$ and $X_{n+1}$ ? 
\item It is possible to estimate the behavior of a particle $X_{r}$ by only considering its neighbouring particles ?
\end{enumerate}
}
The probability space at the equilibrium can be simplified : 
\begin{defn}
A Sawtooth model $\mathcal{S}$ is the data of :
\begin{itemize}
\item $\lbrace \mu_{i},\nu_{i}\rbrace$ a collection of finite measures on $[0,1]$ with respective density functions $\lbrace f_{i},g_{i}\rbrace _{1\leq i\leq n}$, each of them being an increasing $C^{1}$  function on $[0,1]$.
\item A probability space $\Omega(\lbrace f_{i},g_{i}\rbrace)=([0,1]^{n+1}\times[0,1]^{n},\mathbb{P})$ with probability density 
$$d\mathbb{P}(\lbrace x_{i},y_{j}\rbrace)=\frac{1}{\mathcal{V}}\prod \mathbf{1}_{x_{i}\leq y_{i}\geq x_{i+1}}f_{i}(y_{i}-x_{i})g_{i}(y_{i}-x_{i+1}).$$
The quantity $\mathcal{V}$ is called the volume of $\mathcal{S}$ and is sometimes denoted $\mathcal{V}(\mathcal{S})$ to avoid confusion.
\item $2n+1$ random variables $\lbrace X_{i}\rbrace$ and $\lbrace Y_{i}\rbrace$ corresponding to the $2n+1$ coordinates on $[0,1]^{n+1}\times[0,1]^{n}$.
\end{itemize}
$\mathcal{S}$ is said renormalized if each $\mu_{i},\nu_{i}$ is a probability measure.
\end{defn}
If we set $f_{i}(r)=\exp(-\tilde{f}_{i}(r)/(k_{B}T)$ and $g_{i}(r)=\exp(-\tilde{g}_{i}(r)/(k_{b}T)$, we recover the density of \eqref{gibbs}. 
The volume has the following expression: 
\begin{equation}\label{volume}
\mathcal{V}(\mathcal{S})=\int_{[0,1]^{2n+1}}\prod \mathbf{1}_{x_{i}\leq y_{i}\geq x_{i+1}}f_{i}(y_{i}-x_{i})g_{i}(y_{i}-x_{i+1})\prod dx_{i}dy_{i}.
\end{equation}
In particular an appropriate rescaling of the measures $\mu_{i},\nu_{i}$ can transform any Sawtooth model into a normalized one, without changing the probability space. Thus from now on and unless stated otherwise, the model is assumed normalized. In case we are considering non-normalized models, we will use the notation $f_{i},g_{i},$etc. for the normalized quantities, and $\tilde{f}_{i},\tilde{g}_{i},etc.$ for the non-renormalized one.\\
For each subset of particles $A=(q_{i_{1}},\dots,q_{i_{k}},p_{j_{1}},\dots,p_{j_{k'}})$ and measurable event $\mathcal{X}$, denote by
$$d_{A\vert \mathcal{X}}(x_{i_{1}},\dots,x_{i_{k}},y_{j_{1}},\dots,y_{j_{k'}})$$
the marginal density of $A$ conditioned on $\mathcal{X}$. The subscripts will be dropped when there is no confusion, and we denote by $X_{I}$ the first variable $X_{1}$ and $X_{F}$ the last particle $X_{n+1}$.
Finally since the system is fully described by the functions $\lbrace f_{i},g_{j}\rbrace$, we will refer sometimes to a particular system just by mentioning this set of functions.\\
The definition of a Sawtooth model yields directly two first facts. The first result stresses the Markovian aspect of a Sawtooth model :
\begin{lem}\label{markov}
Let $\mathcal{S}$ be a Sawtooth model of size $n$, and $1\leq i_{1}<i_{2},\dots,i_{r}\leq n$ be distinct indices. Then for all $x_{i_{1}},\dots,x_{i_{r}}\in[0,1]$, and $i<i_{1}$,
$$d_{X_{i}\vert X_{i_{1}}=x_{i_{1}},\dots,X_{i_{r}}=x_{i_{r}}}=d_{X_{i}\vert X_{i_{1}}=x_{i_{1}}}.$$
\end{lem}
The proof is a straightforward rephrasing of the density of the model.\\
The second one is a generalization of Lemma $3-(a)$ in \cite{bender2004asymptotics}. :
\begin{lem}\label{monden}
Let $1\leq r\leq n+1$, and let $\mathcal{X}$ be an event depending on the position of all particles except $X_{r}$. Then $d_{X_{r}\vert\mathcal{X}}(x_{r})$ is decreasing in $x_{r}$.
\end{lem}
\begin{proof}
Let $a$ be in $[0,1]$. By Lemma \ref{markov},
\begin{align*}
d_{X_{r}\vert \mathcal{X}}(a)=&\int_{[0,1]^{2}}d_{(X_{r}\vert \mathcal{X})\vert Y_{r-1}=z,Y_{r+1}=z'}(a)d_{Y_{r-1},Y_{r+1}\vert \mathcal{X}}(z,z')dzdz'\\
=&\int_{[0,1]^{2}}d_{X_{r}\vert Y_{r-1}=z,Y_{r+1}=z'}(a)d_{Y_{r-1},Y_{r+1}\vert \mathcal{X}}(z,z')dzdz'.
\end{align*}
Thus it suffices to prove the monotonity in the case of a conditioning on $Y_{r-1}=z,Y_{r+1}=z'$. In this case
$$d_{X_{r}\vert Y_{r-1}=z,Y_{r+1}=z'}(a)=\mathbf{1}_{z\geq a,z'\geq a}\frac{1}{R}(g_{r-1}(z-a)f_{r}(z'-a)),$$
with $R$ a renormalizing constant.
Since since $g_{r-1}$ and $f_{r}$ are increasing, this concludes the proof.
\end{proof}
The same result holds for upper particles, but in this case the density is increasing.\\

\subsection{The processes $\mathcal{S}_{\lambda}$ and $\Sigma_{\lambda}$}
Let us see how these definitions fit into the framework of compositions. The main idea from \cite{ehrenborg2002probabilistic} is to consider the set of all permutations with a given descent set $D_{\lambda}$ as a probability space.\\
$\vert\Sigma_{\lambda}\vert$ can indeed be related to the volume of a polytope in $[0,1]^{n}$ (see for example the survey of Stanley on alternating permutations, \cite{stanley2010survey}) . For each sequence of distincts elements $\vec{\xi}=(\xi_{1},\dots,\xi_{n})$ in $[0,1]$, let $std^{-1}(\vec{\xi})$ (the inverse standardization of $\vec{\xi}$) be the permutation that assigns to each $j$ the index $i_{j}$ in the reordering $(\xi_{i_{1}}<\dots<\xi_{i_{n}})$.
\begin{prop}\label{voldes}
Let $\lbrace x_{i}\rbrace_{1\leq i\leq n}$ be a collection of independent uniform random variables on $[0,1]$. Then the law of $\sigma_{\lambda}$ is the law of $std^{-1}(x_{1},\dots,x_{n})$ conditioned on the fact that $x_{i}>x_{i+1}$ if and only if $i\in D_{\lambda}$. In particular the following expression of the number of permutations with descent set $D_{\lambda}$ holds :
$$\vert\Sigma_{\lambda}\vert =n!\int_{[0,1]^{n}}\prod_{i\in D_{\lambda}}\mathbf{1}_{x_{i}\geq x_{i+1}}\prod_{i\not \in D_{\lambda}}\mathbf{1}_{x_{i}\leq x_{i+1}}\prod dx_{i},$$
with $x_{n+1}=1$.
\end{prop}
The proof of the latter proposition is straightforward as soon as we remark that the volume of the polytope $\lbrace 0\leq x_{1},\dots, x_{n}\leq 1\rbrace$ is exactly $\frac{1}{n!}$. Since the indicator function in the integrand depends on conditions between neighbouring points, this result can be rephrased in terms of Sawtooth model.\\
Regrouping the inequalities between elements of the same run of $\lambda$ yields:
\begin{equation}\label{eqnLoi}
\vert\Sigma_{\lambda}\vert=n!\int_{[0,1]^{n}}\mathbf{1}_{x_{1}\leq x_{2}\leq\dots \leq x_{i_{1}}}\mathbf{1}_{x_{i_{1}}\geq x_{i_{1}+1}\geq \dots\geq x_{i_{1}+i_{2}} }\dots\mathbf{1}_{x_{n-i_{2r}}\leq \dots\leq x_{n}}\prod dx_{i},
\end{equation}
and by integrating over all the coordinates that do not correspond to extreme cells, we get 
\begin{align*}
\vert\Sigma_{\lambda}\vert=&n!\int_{[0,1]^n}\mathbf{1}_{x_{1}^{-}\leq x_{1}^{+}\geq x_{2}^{-}\leq\dots}\frac{1}{(l(s_{1})-2)!}\vert x_{1}^{+}-x_{1}^{-}\vert^{l(s_{1})-2}\\
&\frac{1}{(l(s_{2})-2)!}\vert x_{1}^{+}-x_{2}^{-}\vert^{l(s_{2})-2}\dots \frac{1}{l(s_{2r})-2}\vert x_{r}^{+}-x_{r+1}^{-}\vert^{l(s_{2r})-2}\prod dx_{i}^{+}\prod dx_{i}^{-}.\\
\end{align*}
Let $\mathcal{S}_{\lambda}$ be the non-renormalized Sawtooth model with the non-renormalized density functions $\lbrace \tilde{f}_{j},\tilde{g}_{j}\rbrace_{1\leq i\leq r}$ such that 
$$\tilde{f}_{j}(t)=\frac{1}{(l(s_{2j-1})-2)!}t^{l(s_{2j-1})-2}, \tilde{g}_{j}(t)=\frac{1}{(l(s_{2j})-2)!}t^{l(s_{2j})-2}.$$
A comparison between the latter expression of $\vert \Sigma_{\lambda}\vert $ and the expression \eqref{volume} of the volume of a Sawtooth model gives
$$\vert\Sigma_{\lambda}\vert=\vert \lambda\vert !\mathcal{V}(\mathcal{S}_{\lambda})$$
To sum up, two processes are constructed from $\lambda$. The first one, $\sigma_{\lambda}$ comes from the uniform random standard filling of the ribbon Young tableau $\lambda$, and the second one comes from the construction of an associated model $\mathcal{S}_{\lambda}$. They are of course intimely related, even if the first one is discrete and the second one continuous. $\sigma_{\lambda}$ can be recovered from $\mathcal{S}_{\lambda}$ by the inverse standardization, and when $\vert\lambda\vert$ goes to infinity $(\frac{\sigma_{\lambda}(1)}{n+1},\frac{\sigma_{\lambda}(n)}{n+1})$ and $(X_{I},X_{F})$ are approximately the same :
\begin{lem}\label{NNRPdes}
The following inequality always holds for $0<\epsilon<1, n\in\mathbb{N}$:
$$\mathbb{P}(\sup(\vert \frac{\sigma_{\lambda}(1)}{n+1}-X_{I}\vert,\vert\frac{\sigma_{\lambda}(n+1)}{n}-X_{F}\vert>\frac{A}{\sqrt{n+2}})\leq \frac{2}{A^{2}}$$
In particular if the densities of $x_{I}$ and $x_{F}$ remain bounded by a constant $B$,
$$\Vert F_{(X_{I},X_{F})}-F_{\frac{\sigma(1)}{n+1},\frac{\sigma(n)}{n+1}}\Vert\rightarrow_{\lambda\rightarrow +\infty} 0.$$ 
\end{lem}
\begin{proof}
Let us evaluate $\mathbb{P}(\vert \frac{\sigma_{\lambda}(1)}{n}-X_{I}\vert>\frac{A}{n+2})$. Let condition this on a particular realization $\sigma$ of $\sigma_{\lambda}$, and suppose that $\sigma(1)=k$. In this case, the conditional density of $X_{I}$ is :
\begin{align*}
d_{X_{I}\vert \sigma_{\lambda}=\sigma}(x_{I})=&n!(\int_{0\leq x_{\sigma^{-1}(1)}\leq\dots\leq x_{\sigma^{-1}(k-1)} \leq x_{1}}\prod_{1\leq \sigma(i)\leq k-1}dx_{i})\\
&\quad\quad\quad\quad(\int_{x_{1}\leq x_{\sigma^{-1}(k+1)}\leq\dots\leq x_{\sigma^{-1}(n)}\leq 1}\prod_{k+1\leq \sigma(i)\leq 1}dx_{i})\\
=&\frac{n!}{(k-1)!(n-k)!}x_{I}^{k-1}(1-x_{I})^{n-k}.
\end{align*}
Computing the conditional expectation yields $\mathbb{E}(X_{I}\vert \sigma)=\frac{k}{n+1}$ and 
$$Var(x_{I}\vert \sigma)=(\frac{k}{n+1}\frac{n+1-k}{n+1})\frac{1}{n+1}\leq \frac{1}{n+2}.$$
Thus by the Chebyshev's inequality,
$$\mathbb{P}_{X_{I}\vert \sigma_{\lambda}=\sigma}(\vert X_{I}-\frac{\sigma(1)}{n+1}\vert>\frac{A}{\sqrt{n+2}})\leq \frac{1}{A^{2}}.$$
Integrating this inequality on all the disjoint events $\sigma$ on which $X_{I}$ can be conditioned yields the first fart of the Lemma. The second part is straightforward.
\end{proof}
In the sequel let $\tilde{\gamma}_{r}$ denote for $r\geq 2$ the function $\tilde{\gamma}_{r}(t)=\frac{1}{(r-2)!}t^{r-2}$, and $\gamma_{r}(t)=(r-1)t^{r-2}$ its renormalized density function.
\section{Convex Sawtooth Model}
\subsection{Log-concave densities}
To be able to get some results on the behavior of the particles, it is necessary to impose some conditions on the density functions $\lbrace f_{i},g_{i}\rbrace $. Actually the condition we need is quite natural from a physical point of view, since we will require that the repulsive forces in the definition of the Sawtooth model come from a convex potential : the consequence is that the density functions should be log-concave. This motivates the following definition :
\begin{defn}
A Sawtooth model is called convex if all the functions $(f_{i},g_{i})_{1\leq i\leq n}$ are log-concave. This means that for all $1\leq i\leq n$, $\frac{f_{i}'(t)}{f_{i}(t)}$ and $\frac{g_{i}'(t)}{g_{i}(t)}$ are decreasing.
\end{defn}
The main advantage of the log-concavity is that the behavior of the particles becomes monotone in a certain sense. For $1\leq s\leq n+1$ denote by $\mathcal{S}_{\rightarrow P_{s}}$ (resp. $\mathcal{S}_{P_{s}\leftarrow}$) the Sawtooth model obtained by keeping only the particles and interactions between $X_{I}$ and $P_{s}$ (resp. $P_{s}$ and $X_{F}$).
\begin{prop}\label{monoton}
Let $\lbrace f_{i},g_{i}\rbrace$ be a convex Sawtooth model. Then for $1\leq s\leq n$, $0\leq t\leq 1$, 
$F_{X_{s}\vert Y_{s}=y}(t)$ is decreasing in $y$, and  $F_{Y_{s}\vert X_{s+1}=x}(t)$  is decreasing in $x$. Moreover 
$$F_{X_{s}\vert Y_{s}=y}(t)\geq F_{X_{s}\vert \mathcal{S}_{\rightarrow X_{s}}}(t)$$
and 
$$F_{Y_{s}\vert X_{s+1}=x}(t)\leq F_{Y_{s}\vert \mathcal{S}_{\rightarrow Y_{s}}}(t)$$
\end{prop}
\begin{proof}
Let $d(x)$ be the density of $X_{s}$ in $\mathcal{S}_{\rightarrow X_{s}}$. Then by the definition of the propability density of $\mathcal{S}$, the density of $X_{s}$ in $\mathcal{S}$ conditioned on the value of $Y_{s}$ is $\mathbf{1}_{x\leq y}\frac{d(x)f_{s}(y-x)}{A}$, with $A$ a normalizing constant. Thus the cumulative distribution function $F_{y}(.)$ of $X_{s}$ conditioned on $Y_{s}=y$ is 
$$F_{y}(t)=\frac{\int_{0}^{t\wedge y}d(x)f_{s}(y-x)dx}{\int_{0}^{y}d(x)f_{s}(y-x)dx}.$$
For $t> y$ it is clear that $\frac{\partial}{\partial y}F_{y}(t)=0$, and from now on we only consider $t\leq y$. Since the logarithm function is increasing, it is enough to show that $\frac{\partial}{\partial y}\log(F_{y}(t))\leq 0$. This derivative is equal to
$$\frac{\partial}{\partial y}\log(F_{y}(t))=\frac{\int_{0}^{t}d(x)f_{s}'(y-x)dx}{\int_{0}^{t}d(x)f_{s}(y-x)dx}-\frac{\int_{0}^{y}d(x)f_{s}'(y-x)dx}{\int_{0}^{y}d(x)f_{s}(y-x)dx}-\frac{d(y)f_{s}(0)}{\int_{0}^{y}d(x)f_{s}(y-x)dx}.$$
Since $(-\frac{d(y)f_{s}(0)}{\int_{0}^{y}d(x)f_{s}(y-x)dx})\leq 0$, the non-positivity of the remaining part of the sum suffices. Denote
$$\Delta =\int_{0}^{t}d(x)f_{s}'(y-x)dx\int_{0}^{y}d(x)f_{s}(y-x)dx-\int_{0}^{y}d(x)f_{s}'(y-x)dx\int_{0}^{t}d(x)f_{s}(y-x)dx.$$
Thus we have to show that $\Delta\leq 0$. For $t\leq y$,
\begin{align*}
\Delta=\int_{0}^{t}&d(x)f_{s}'(y-x)dx\left(\int_{0}^{t}d(x)f_{s}(y-x)dx+\int_{t}^{y}d(x)f_{s}(y-x)dx\right)\\
&-\left(\int_{0}^{t}d(x)f_{s}'(y-x)dx+\int_{t}^{y}d(x)f_{s}'(y-x)dx\right)\int_{0}^{t}d(x)f_{s}(y-x)dx\\
=\int_{0}^{t}&d(x)f_{s}'(y-x)dx\int_{t}^{y}d(x)f_{s}(y-x)dx\\
&-\int_{t}^{y}d(x)f_{s}'(y-x)dx\int_{0}^{t}d(x)f_{s}(y-x)dx.\\
\end{align*}
Expressing products of integrals as double integrals yields
\begin{align*}
\Delta=&\int_{\substack{0\leq z_{1}\leq t\\t\leq z_{2}\leq y}}d(z_{1})d(z_{2})f_{s}'(y-z_{1})f_{s}(y-z_{2})dz_{1}dz_{2}\\
&\qquad\qquad\qquad\qquad-\int_{\substack{0\leq z_{1}\leq t\\t\leq z_{2}\leq y}}d(z_{1})d(z_{2})f_{s}(y-z_{1})f_{s}'(y-z_{2})dz_{1}dz_{2}\\
=&\int_{\substack{0\leq z_{1}\leq t\\t\leq z_{2}\leq y}}d(z_{1})d(z_{2})(f_{s}'(y-z_{1})f_{s}(y-z_{2})-f_{s}(y-z_{1})f_{s}'(y-z_{2}))dz_{1}dz_{2}.\\
\end{align*}
Since $d(z_{1})d(z_{2})$ is positive and $\frac{f_{s}'(t)}{f_{s}(t)}$ is decreasing, $\Delta\leq 0$ and the first part of the Proposition is proven.\\
From the first part of the Proposition, it suffices to prove the first inequality of the second part only for $y=1$. Since $f_{s}$ is increasing, there exists a measure $\mu$ on $[0,1]$ such that $f_{s}(x)=\int_{0}^{x}d\mu(u)$. Thus
$$F_{1}(t)=\frac{\int_{0}^{t}d(x)(\int_{0}^{1-x}d\mu(u))dx}{\int_{0}^{1}d(x)(\int_{0}^{1-x}d\mu(u))dx}=\frac{\int_{[0,1]^{2}}\mathbf{1}_{x\leq t,u\leq 1-x}d(x)d\mu(u)dx}{\int_{[0,1]^{2}}\mathbf{1}_{u\leq 1-x}d(x)d\mu(u)dx}.$$
The main point is to express the latter quantity as the expectation of a random variable almost surely greater than $\int_{0}^{t}d(x)dx$. Interverting the integrals yields
\begin{align*}
F_{1}(t)=&\frac{\int_{0}^{1}\left(\int_{0}^{t\wedge (1-u)}d(x)dx\right)d\mu(u)}{\int_{0}^{1}\left(\int_{0}^{1-u}d(x)dx\right)d\mu(u)}
=\frac{\int_{0}^{1}\left(\int_{0}^{t\wedge (1-u)}\frac{d(x)}{\int_{0}^{1-u}d(x)dx}dx\right)(\int_{0}^{1-u}d(x)dx)d\mu(u)}{\int_{0}^{1}\left(\int_{0}^{1-u}d(x)dx\right)d\mu(u)}.
\end{align*}
Let $\tilde{U}$ be a random variable absolutely continuous with respect to $\mu$ and having the density 
$$d_{\tilde{U}}(u)=\frac{\left(\int_{0}^{1-u}d(x)dx\right)d\mu(u)}{\int_{0}^{1}\left(\int_{0}^{1-u}d(x)dx\right)d\mu(u)}.$$
Then 
$$F_{1}(t)=\mathbb{E}_{\tilde{U}}(\frac{\int_{0}^{t\wedge (1-\tilde{U})}d(x)dx}{\int_{0}^{1-\tilde{U}}d(x)dx}).$$
Since for each $u\geq 0$
$$\frac{\int_{0}^{t\wedge 1-u}d(x)dx}{\int_{0}^{1-u}d(x)dx}\geq \int_{0}^{t}d(x)dx,$$
this concludes the proof.\\
It is exactly the same for $F_{Y_{s}\vert X_{s+1}=x}(t)$.
\end{proof}

\subsection{Alternating pattern of a convex sawtooth model}
Proposition \ref{monoton} yields two main features for the model. The first one is an extension of the previous result.
\begin{prop}\label{monoton2}
Let $1\leq s\leq r$, $0\leq t\leq 1$. Then $F_{X_{s}\vert X_{r}=x}(t)$ is decreasing in $x$ and $F_{X_{s}\vert Y_{r}=y}(t)$ is decreasing in $y$. Moreover 
$$F_{X_{s}\vert \mathcal{S}_{\rightarrow X_{r}}}(t)\leq F_{X_{s}\vert Y_{r}=y}(t)$$
and
$$F_{X_{s}\vert \mathcal{S}_{\rightarrow Y_{r}}}(t)\geq F_{X_{s}\vert X_{r+1}=x}(t).$$
\end{prop}
\begin{proof}
Let $s\geq 1$ and let us prove the monotonicty by recurrence on $r$, starting at $s=r$. $F_{X_{s}\vert X_{s}=x}(t)$ is clearly decreasing in $x$ and from Proposition \ref{monoton}, $F_{X_{s}\vert Y_{s}=y}(t)$ is decreasing in $y$. Thus the initialization is done.\\
Suppose the result proved until $X_{r}$. Then 
$$F_{X_{s}\vert X_{r+1}=x}(t)=\int_{0}^{1} F_{X_{s}\vert Y_{r}=y,X_{r+1}=x}(t)d_{Y_{r}\vert X_{r+1}=x}(y)dy,$$
and by an integration by part, since from Lemma \ref{markov} $ F_{X_{s}\vert Y_{r}=y,X_{r+1}=x}(t)=F_{X_{s}\vert Y_{r}=y}(t)$,
$$F_{X_{s}\vert X_{r+1}=x}(t)=F_{X_{s}\vert Y_{r}=1}(t)-\int_{0}^{1} \frac{\partial}{\partial y}F_{X_{s}\vert Y_{r}=y}(t)F_{Y_{r}\vert X_{r+1}=x}(y)dy.$$
Thus 
$$\frac{\partial}{\partial x}F_{X_{s}\vert X_{r+1}=x}(t)=-\int_{0}^{1} \frac{\partial}{\partial y}F_{X_{s}\vert Y_{r}=y}(t)\frac{\partial}{\partial x }F_{Y_{r}\vert X_{r+1}=x}(y)dy.$$
By recurrence $\frac{\partial}{\partial y}F_{X_{s}\vert Y_{r}=y}(t)$ is negative and by Proposition \ref{monoton} $\frac{\partial}{\partial x }F_{Y_{r}\vert X_{r+1}=x}(x)$ is negative, thus $\frac{\partial}{\partial x}F_{X_{s}\vert X_{r+1}=x}(t)$ is also negative. It is exactly the same for $F_{X_{s}\vert Y_{r+1}=y}(t)$.\\
Let us prove the second part of the proposition and let $y\in [0,1]$. Conditioning $X_{s}$ on $X_{r}$ in $\mathcal{S}_{\rightarrow X_{r}}$ yields 
$$F_{X_{s}\vert \mathcal{S}_{\rightarrow X_{r}}}(t)=\mathbb{E}(F_{X_{s}\vert X_{r}=\tilde{X}_{r}}(t)),$$
with $\tilde{X}_{r}$ following the law of $q_{r}$ in $\mathcal{S}_{\rightarrow X_{r}}$.\\
On one hand from the first part of the proposition, $F_{X_{s}\vert X_{r}=x}(t)$ is decreasing in $x$. On the other hand from Proposition \ref{monoton}, $\tilde{X}_{r}$ stochastically dominates $(X_{r}\vert Y_{r}=y)$. Thus from Proposition \ref{definitionStochastDom}, 
$$F_{X_{s}\vert \mathcal{S}_{\rightarrow X_{r}}}(t)=\mathbb{E}(F_{X_{s}\vert X_{r}=\tilde{X}_{r}}(t))\leq F_{X_{s}\vert Y_{r}=y}(t).$$
The same pattern proves the second inequality.
\end{proof}
There is an immediate consequence of this Proposition on the behavior of $F_{X_{s}\vert \mathcal{S}_{\rightarrow X_{n}}}(t)$ with $n\geq s$.
\begin{cor}
The following inequalities hold for $n\geq s$:
$$F_{X_{s}\vert \mathcal{S}_{\rightarrow X_{s}}}(t)\leq \dots\leq F_{X_{s}\vert \mathcal{S}_{\rightarrow X_{n}}}(t)\leq\dots\leq F_{X_{s}\vert \mathcal{S}_{\rightarrow Y_{n}}}(t)\dots \leq F_{X_{s}\vert \mathcal{S}_{\rightarrow Y_{s}}}(t).$$
\end{cor}
\begin{proof}
The previous Proposition yields directly the following inequalities :
$$F_{X_{s}\vert \mathcal{S}_{\rightarrow Y_{r}}}(t)\geq F_{X_{s}\vert Y_{r}=1}\geq F_{X_{s}\vert\mathcal{S}_{\rightarrow X_{r}}}(t).$$
Moreover 
\begin{align*}
F_{X_{s}\vert \mathcal{S}_{\rightarrow X_{n+1}}}(t)=&\int_{[0,1]}F_{X_{s}\vert  Y_{n}=y}(t)d_{Y_{n}\vert \mathcal{S}_{\rightarrow X_{n+1}}}(y)dy\\
\geq &\int_{[0,1]}F_{X_{s}\vert \mathcal{S}_{\rightarrow X_{n}}}(t)d_{Y_{n}\vert \mathcal{S}_{\rightarrow X_{n+1}}}(y)dy\\
\geq & F_{X_{s}\vert \mathcal{S}_{\rightarrow X_{n}}}(t),
\end{align*}
the first inequality being due to Proposition \ref{monoton}. By symmetry between $X_{n}$ and $Y_{n}$ the general result holds.
\end{proof}

\subsection{Estimates on the behavior of extreme particles}
As a second consequence of Proposition \ref{monoton} we can get a more accurate estimate on the behavior of the first and last particles of $\mathcal{S}$. In particular we can achieve a coupling of $(X_{I},X_{F})$ with two couples of random variables, which only depend on $f_{1}$ and $g_{n}$ and give some bounds on $(X_{I},X_{F})$ in the sense of the stochastic domination.\\
In this paragraph we will not assume that the first and last particles are lower ones, and deal with model of any type (refer to Remark \ref{typeModel} for the definition of the type of a model).
Moreover to describe the bounding random variables we introduce two particular transforms $\Gamma^{+}$ and $\Gamma^{-}$:
\begin{defn}
Let $f$ be a positive function on $[0,1]$. Then $\Gamma^{+}(f)$ and $\Gamma^{-}(f)$ are the functions defined on $[0,1]$ as :
$$\Gamma^{-}(f)(t)=\frac{\int_{1-t}^{1}f(u)du}{\int_{0}^{1}f(u)du},$$
and
$$\Gamma^{+}(f)(t)=\frac{\int_{0}^{t}f(u)du}{\int_{0}^{1}f(u)du}.$$
\end{defn}
Remark that $\Gamma^{-}(f)(t)$ (resp. $\Gamma^{+}(f)(t)$) is the cumulative distribution function of the random variable $1-Z$ (resp. $Z$), $Z$ being the random variable with density $\frac{f(x)}{\int_{0}^{1}f(x)dx}$.
\begin{prop}\label{domNNRPinit}
Let $\mathcal{S}$ be a convex Sawtooth model of type $\epsilon$ with density functions $\lbrace f_{i},g_{i}\rbrace_{1\leq i\leq n}$ and at least four particles. There exists a probability space and two couples of random variables $(X_{+},Y_{+}),(X_{-},Y_{-})$ on it, such that :
\begin{itemize}
\item $(X_{-},Y_{-})\preceq_{\epsilon} (X_{I},X_{F})\preceq_{\epsilon} (X_{+},Y_{+}).$
\item $X_{+}$ and $Y_{+}$ are independent with distribution function 
$$F_{X_{+},Y_{+}}(s,t)=\Gamma^{\epsilon_{1}}(f_{1})(s)\Gamma^{\epsilon_{2}}(g_{n})(t).$$
\item $X_{-}$ and $Y_{-}$ are independent with distribution function 
$$F_{X_{-},Y_{-}}(s,t)=\left(\Gamma^{\epsilon_{1}}\circ \Gamma^{\epsilon_{1}^{*}}(f_{1})\right)(s)\left(\Gamma^{\epsilon_{2}}\circ\Gamma^{\epsilon_{2}^{*}}(g_{n})\right)(t).$$
\end{itemize}
with $-^{*}=+$ and $+^{*}=-$.
\end{prop}
\begin{proof}
We assume without loss of generality that each $f_{i},g_{i}$ is renormalized and, since the type of the Sawtooth model doesn't change the pattern of the proof, we assume that $\mathcal{S}$ is of type $--$.\\
On one hand the conditional law of $(X_{I},X_{F})$ given the value of $Y_{1}=y_{1},Y_{n}=y_{n}$ has for cumulative distribution function :
\begin{align*}
F_{X_{I},X_{F}\vert Y_{1}=y_{1},Y_{n}=y_{n}}(t_{1},t_{2})=&\frac{(\int_{0}^{t_{1}\wedge y_{1}}f_{1}(y_{1}-x)dx)(\int_{0}^{t_{2}\wedge y_{n}}g_{n}(y_{n}-y)dy)}{(\int_{0}^{y_{1}}f_{1}(x)dx)(\int_{0}^{y_{n}}g_{n}(x)dx)}\\
=&F_{X_{I}\vert Y_{1}=y_{1}}(t)F_{X_{F}\vert Y_{n}=y_{n}}(t).\\
\end{align*}
This together with Proposition \ref{monoton} gives the bound
\begin{align*}
F_{X_{1},X_{n+1}\vert Y_{1}=y_{1},Y_{n}=y_{n}}(t_{1},t_{2})=&F_{X_{I}\vert Y_{1}=y_{1}}(t)F_{X_{F}\vert Y_{n}=y_{n}}(t)\\
\geq& F_{X_{I}\vert Y_{1}=1}(t)F_{X_{F}\vert Y_{n}=1}(t).
\end{align*}
Since 
$$F_{X_{I}\vert Y_{1}=1}(s)F_{X_{F}\vert Y_{n}=1}(t)=(1-F_{f_{1}}(1-s))(1-F_{g_{n}}(1-t_{2}))=\Gamma^{-}(f_{1})(s)\Gamma^{-}(g_{n})(t),$$
this gives the upper part of the stochastic bound.\\
On the other hand, the density of $(Y_{1},Y_{n})$ conditioned on the value of $(X_{2},X_{n})$ is 
\begin{align*}
&d_{Y_{1},Y_{n}\vert X_{2}=x_{2},X_{n}=x_{n}}(y_{1},y_{n})\\
=&\mathbf{1}_{y_{1}\geq x_{2},y_{n}\geq x_{n}}\frac{(\int_{0}^{y_{1}}f_{1}(y_{1}-x)dx)g_{1}(y_{1}-x_{2})}{\int_{x_{2}}^{1}(\int_{0}^{z}f_{1}(z-x)dx)g_{1}(z-x_{2})dz}\frac{(\int_{0}^{y_{n}}g_{n}(y_{n}-x)dx)f_{n}(y_{n}-x_{n})}{\int_{x_{n}}^{1}(\int_{0}^{z}g_{n}(z-x)dx)f_{n}(z-x_{n})dz}\\
=&\mathbf{1}_{y_{1}\geq x_{2},y_{n}\geq x_{n}}\frac{F_{f_{1}}(y_{1})g_{1}(y_{1}-x_{2})}{\int_{x_{2}}^{1}F_{f_{1}}(z)g_{1}(z-x_{2})dz}\frac{F_{g_{n}}(y_{n})f_{n}(y_{2}-x_{n})}{\int_{x_{n}}^{1}F_{g_{n}}(z)f_{n}(z-x_{n})dz}.\\
\end{align*}
Factorizing the latter density yields
$$d_{Y_{1},Y_{n}\vert X_{2}=x_{2},X_{n}=x_{n}}(y_{1},y_{n})=d_{Y_{1}\vert X_{2}=x_{2}}(y_{1})d_{Y_{n}\vert X_{n}=x_{n}}(y_{n}).$$
Let us first consider $Y_{1}$. Recall that $g_{1}$ is an increasing $\mathcal{C}^{1}$ function. This means in particular that 
$$g_{1}(x)=\frac{1}{K}\int_{0}^{x}d\lambda(u),$$
with $\lambda$ a probability measure on $[0,1]$ having eventually a dirac mass at $0$ and then a continuous density function on $]0,1]$. Thus the density of $Y_{1}$ conditioned on the value of $X_{2}$ is 
$$d_{Y_{1}\vert X_{2}=x_{2}}(y_{1})=\frac{1}{A}\mathbf{1}_{y_{1}\geq x_{2}}F_{f_{1}}(y_{1})\int_{x_{2}}^{y_{1}}d\lambda(u-x_{2}),$$
with $A$ a normalizing constant. Let $d_{u}$ be the density function defined for $0\leq u\leq 1$ by 
$$d_{u}(y)=\frac{1}{A_{u}}\mathbf{1}_{y\geq u}F_{f_{1}}(y_{1}),$$
with $A_{u}$ a normalizing constant depending on $u$ and let $F_{u}(t)$ be the associated cumulative distribution function. On one hand 
\begin{align*}
F_{Y_{1}\vert X_{2}=x_{2}}(t)=&\frac{\int_{0}^{t}\mathbf{1}_{y_{1}\geq x_{2}}F_{f_{1}}(y_{1})\int_{x_{2}}^{y_{1}}d\lambda(u-x_{2})dy_{1}}{\int_{0}^{1}\mathbf{1}_{y_{1}\geq x_{2}}F_{f_{1}}(y_{1})\int_{x_{2}}^{y_{1}}d\lambda(u-x_{2})dy_{1}}\\
=&\frac{\int_{0}^{t}\int_{x_{2}}^{1}\mathbf{1}_{y_{1}\geq u}F_{f_{1}}(y_{1})d\lambda(u-x_{2})dy_{1}}{\int_{0}^{1}\int_{x_{2}}^{1}\mathbf{1}_{y_{1}\geq u}F_{f_{1}}(y_{1})d\lambda(u-x_{2})dy_{1}},
\end{align*}
and after interverting the integrals, since $F_{u}(1)=1$,
\begin{align*}
F_{Y_{1}\vert X_{2}=x_{2}}(t)=&\frac{\int_{x_{2}}^{1}(\int_{0}^{t}\mathbf{1}_{y_{1}\geq u}F_{f_{1}}(y_{1})dy_{1})d\lambda(u-x_{2})}{\int_{x_{2}}^{1}(\int_{0}^{1}\mathbf{1}_{y_{1}\geq u}F_{f_{1}}(y_{1})dy_{1})d\lambda(u-x_{2})}\\
=&\frac{\int_{x_{2}}^{1}A_{u}F_{u}(t)d\lambda(u-x_{2})}{\int_{x_{2}}^{1}A_{u}d\lambda(u-x_{2})}\\
=&\mathbb{E}_{\tilde{U}}(F_{\tilde{U}}(t)),
\end{align*}
with $\tilde{U}$ a random variable with law $d\tilde{U}(u)=\mathbf{1}_{u\geq x_{2}}\frac{A_{u}d\lambda(u-x_{2})}{\int_{x_{2}}^{1}A_{u}d\lambda(u-x_{2)}}$.\\
On the other hand
$$F_{u}(t)=\mathbf{1}_{t\geq u}\frac{\int_{u}^{t}F_{f_{1}}(u)du }{\int_{u}^{1}F_{f_{1}}(u)du}=\mathbf{1}_{t\geq u}\frac{\mathcal{F}_{f_{1}}(t)-\mathcal{F}_{f_{1}}(u)}{\mathcal{F}_{f_{1}}(1)-\mathcal{F}_{f_{1}}(u)},$$
with $\mathcal{F}_{f_{1}}$ being the primitive of $F_{f_{1}}$ taking the value $0$ at $0$. This yields
\begin{align*}
\frac{\partial}{\partial u} F_{u}(t)=&\frac{\partial}{\partial u}(\mathbf{1}_{u\leq t}\frac{\mathcal{F}_{f_{1}}(t)-\mathcal{F}_{f_{1}}(u)}{\mathcal{F}_{f_{1}}(1)-\mathcal{F}_{f_{1}}(u)})\\
=&\mathbf{1}_{u\leq t}\frac{\partial}{\partial u}((\mathcal{F}_{f_{1}}(t)-\mathcal{F}_{f_{1}}(1))\frac{1}{\mathcal{F}_{f_{1}}(1)-\mathcal{F}_{f_{1}}(u)}+1)\\
=&\mathbf{1}_{u\leq t}(\mathcal{F}_{f_{1}}(t)-\mathcal{F}_{f_{1}}(1))\frac{\partial}{\partial u}(\frac{1}{\mathcal{F}_{f_{1}}(1)-\mathcal{F}_{f_{1}}(u)})\\
=&\mathbf{1}_{u\leq t}(\mathcal{F}_{f_{1}}(t)-\mathcal{F}_{f_{1}}(1))\frac{F_{f_{1}}(u)}{(\mathcal{F}_{f_{1}}(1)-\mathcal{F}_{f_{1}}(u))^{2}}\leq 0,\\
\end{align*}
and thus
$$ F_{u}(t)\leq  F_{0}(t)=\frac{\mathcal{F}_{f_{1}}(t)}{\mathcal{F}_{f_{1}}(1)}.$$
Integrating with respect to $\tilde{U}$ yields  
$$F_{Y_{1}\vert X_{2}=x_{2}}(t)=\mathbb{E}_{\tilde{U}}(F_{\tilde{U}}(t))\leq\mathbb{E}_{\tilde{U}}(F_{0}(t)),$$
and finally $F_{Y_{1}\vert X_{2}=x_{2}}(t)\leq \frac{\mathcal{F}_{f_{1}}(t)}{\mathcal{F}_{f_{1}}(1)}$.
We can now integrate this inequality to get a bound on the cumulative distribution function of $X_{I}$ conditionned on $X_{2}$ :
\begin{align*}
F_{X_{I}\vert X_{2}=x_{2}}(t)=&\int_{0}^{1} F_{X_{I}\vert Y_{1}=y}(t)d_{Y_{1}\vert X_{2}=x_{2}}(y)dy\\
=&F_{X_{I}\vert Y_{1}=1}(t)-\int_{0}^{1}\frac{\partial}{\partial y}F_{X_{I}\vert Y_{1}=y}(t)F_{Y_{1}\vert X_{2}=x_{2}}(y)dy\\
\leq &F_{X_{I}\vert Y_{1}=1}(t)-\int_{0}^{1}\frac{\partial}{\partial y}F_{X_{I}\vert Y_{1}=y}(t)\frac{\mathcal{F}_{f_{1}}(y)}{\mathcal{F}_{f_{1}(1)}}dy\\
\leq &\int_{0}^{1}F_{X_{I}\vert Y_{1}=y}(t)\frac{F_{f_{1}}(y)}{\mathcal{F}_{f_{1}(1)}}dy.
\end{align*}
Note that the sense of the inequality on the third line is due to the negative sign of $\frac{\partial}{\partial y}F_{X_{I}\vert Y_{1}=y}(t)$. Since 
\begin{align*}
\int_{0}^{1}F_{X_{I}\vert Y_{1}=y}(t)\frac{F_{f_{1}}(y)}{\mathcal{F}_{f_{1}(1)}}dy=&\int_{0}^{1}\frac{\int_{0}^{t\wedge y}f_{1}(y-u)du}{F_{f_{1}}(y)}\frac{F_{f_{1}}(y)}{\mathcal{F}_{f_{1}}(1)}dy\\
=&\int_{0}^{t}\int_{u}^{1}\frac{f_{1}(y-u)}{\mathcal{F}_{f_{1}}(1)}dydu\\
=&\frac{\int_{0}^{t}F_{f_{1}}(1-u)du}{\mathcal{F}_{f_{1}}(1)}=\Gamma^{-}(F_{f_{1}})(t),\\
\end{align*}
this yields the inequality
$$F_{X_{I}\vert X_{2}=x_{2}}(t)\leq\Gamma^{-}\circ\Gamma^{+}(f_{1})(t).$$
Note that the latter inequality is valid even if the model has only three particles (see the next Corollary).
Finally since in our case there are at least four particles, $X_{F}\not =X_{2}$, and thus $F_{X_{I}\vert X_{2}=x_{2},X_{F}=y}(t)=F_{X_{1}\vert X_{2}=x_{2}}(t)$. Therefore
$$F_{X_{I}\vert X_{F}=y}(t)\leq \Gamma^{-}\circ\Gamma^{+}(f_{1})(t),$$
and by averaging on $y$,
$$F_{X_{I}}(t)\leq \Gamma^{-}\circ\Gamma^{+}(f_{1})(t).$$
Doing the same with $X_{F}$ gives the bound :
$$F_{X_{F}}(t)\leq\Gamma^{-}\circ\Gamma^{+}(g_{n})(t).$$
The result follows from Lemma \ref{stochasticDominationCondition}.
\end{proof}
In particular as a corollary of the latter proposition (and as a corollary of the proof in the case $n=2$), the following result holds :
\begin{cor}\label{domNNRP}
Let $\mathcal{S}$ be a convex Sawtooth model of type $\epsilon$ with density functions  $\lbrace f_{i},g_{i}\rbrace_{1\leq i\leq n}$ .  There exists a couple of random variables $(Z^{(1)},Z^{(2)})$ such that for $y\in[0,1]$,
\begin{itemize}
\item $Z^{(1)}\preceq_{\epsilon(1)} (X_{I}\vert X_{F}=y)\preceq_{\epsilon(1)} Z^{(2)},$
\item The cumulative distribution function of $Z^{(2)}$ is :
$$F_{Z^{(2)}}(t)=\Gamma^{\epsilon(1)}(f_{1})(t).$$
\item The cumulative distribution function of $Z^{(1)}$ is
$$F_{Z^{(1)}}(t)=\Gamma^{\epsilon(1)}\circ \Gamma^{\epsilon(1)^{*}}(f_{1})(t).$$
\end{itemize}
\end{cor}
\begin{proof}
For $n\geq 3$, the result is deduced from the latter Proposition. In the case $n=2$, the proof is exactly the same as in the latter Proposition, except that we only deal with the left case, and thus we don't need anymore the fact that $X_{2}\not =X_{F}$.
\end{proof}

\section{The independence theorem in a bounded Sawtooth Model}

\subsection{Decorrelation principle and bounding Lemmas}
This section is devoted to the proof of the independence of $X_{I}$ and $X_{F}$ when the number of particles grows whereas the repulsion forces remain bounded. 
\begin{defn}
Let $A>0$. A Sawtooth model $\mathcal{S}$ with density functions $\lbrace f_{i},g_{i}\rbrace$ is bounded by $A$ if 
$$\sup (\Vert f_{i}\Vert_{[0,1]},\Vert g_{i}\Vert_{[0,1]})\leq A.$$
\end{defn}
The purpose is to prove the following Theorem :
\begin{thm}\label{indepPart}
Let $A>0$. For all $\epsilon>0$ there exists $N_{A}\geq 0$ such that for all Sawtooth model $\mathcal{S}$ bounded by $A$ and with $2n\geq N_{A}$ particles we have :
$$\Vert d_{X_{I},X_{F}}(x,y)-d_{X_{I}}(x)d_{X_{F}}(y)\Vert_{\infty}\leq \epsilon$$
\end{thm}
The pattern of the proof is the following : conditioned on the fact that a particle $P$ - from now on called a splitting particule - is closed to the boundary of the domain, the left part $\mathcal{S}_{\rightarrow P}$ and the right part $\mathcal{S}_{\leftarrow P}$ of the system are almost not correlated anymore (see Figure \ref{fig4}).
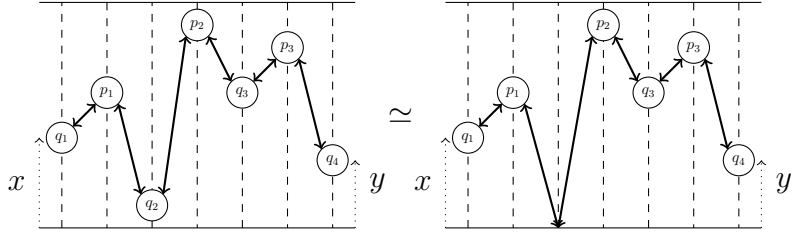
\begin{figure}[!h]

\begin{tikzpicture}[scale=0.3][center]
\begin{scope}
\draw (0,0) -- (14,0);
\draw (0,10)-- (14,10);
\draw [dashed] (1,0) to (1,10);
\draw [dashed] (5,0) to (5,10);
\draw [dashed] (9,0) to (9,10);
\draw [dashed] (3,10) to (3,0);
\draw [dashed] (7,10) to (7,0);
\draw [dashed] (11,10) to (11,0);
\draw [dashed] (13,10) to (13,0);
\node [draw] [circle,fill=white,scale=0.5](q1) at (1,4) {$q_{1}$};
\node [draw] [circle,fill=white,scale=0.5](q2) at (5,1) {$q_{2}$};
\node [draw] [circle,fill=white,scale=0.5](q3) at (9,6) {$q_{3}$};
\node [draw] [circle,fill=white,scale=0.5](q4) at (13,3) {$q_{4}$};
\node [draw] [circle,fill=white,scale=0.5](p1) at (3,6) {$p_{1}$};
\node [draw] [circle,fill=white,scale=0.5](p2) at (7,9) {$p_{2}$};
\node [draw] [circle,fill=white,scale=0.5](p3) at (11,8) {$p_{3}$};

\draw [<->,thick](q1.north east) to (p1.south west);
\draw [<->,thick](p1.south east) to (q2.north west);
\draw [<->,thick](q2.north east) to (p2.south west);
\draw [<->,thick](p2.south east) to (q3.north west);
\draw [<->,thick](q3.north east) to (p3.south west);
\draw [<->,thick](p3.south east) to (q4.north west);
\draw[dotted,->] (0,0)--(0,4);
\node at (-1,2) {$x$};
\draw [dotted, ->] (14,0)--(14,3);
\node at (15,2) {$y$};
\end{scope}
\begin{scope}[shift={(16,0)}]
\node at (0,5) {$\simeq$};
\end{scope}
\begin{scope}[shift={(18,0)}]
\draw (0,0) -- (14,0);
\draw (0,10)-- (14,10);
\draw [dashed] (1,0) to (1,10);
\draw [dashed] (5,0) to (5,10);
\draw [dashed] (9,0) to (9,10);
\draw [dashed] (3,10) to (3,0);
\draw [dashed] (7,10) to (7,0);
\draw [dashed] (11,10) to (11,0);
\draw [dashed] (13,10) to (13,0);
\node [draw] [circle,fill=white,scale=0.5](q1) at (1,4) {$q_{1}$};
\node [draw] [circle,fill=white,scale=0.5](q3) at (9,6) {$q_{3}$};
\node [draw] [circle,fill=white,scale=0.5](q4) at (13,3) {$q_{4}$};
\node [draw] [circle,fill=white,scale=0.5](p1) at (3,6) {$p_{1}$};
\node [draw] [circle,fill=white,scale=0.5](p2) at (7,9) {$p_{2}$};
\node [draw] [circle,fill=white,scale=0.5](p3) at (11,8) {$p_{3}$};

\draw [<->,thick](q1.north east) to (p1.south west);
\draw [<->,thick](p1.south east) to (5,0);
\draw [<->,thick](5,0) to (p2.south west);
\draw [<->,thick](p2.south east) to (q3.north west);
\draw [<->,thick](q3.north east) to (p3.south west);
\draw [<->,thick](p3.south east) to (q4.north west);
\draw[dotted,->] (0,0)--(0,4);
\node at (-1,2) {$x$};
\draw [dotted, ->] (14,0)--(14,3);
\node at (15,2) {$y$};
\end{scope}
\end{tikzpicture}
\caption{\label{fig4}Decorrelation of the process}
\end{figure}\\
However we may still not have independence if the law of $X_{I}$ and $X_{F}$ depends on which particle splits the system. Thus we have to find a set of particles that is large enough, so that with probability close to one an element of this set is close to the boundary, and such that nonetheless conditioning on having any particle from this set closed to the boundary yields the same law on $(X_{I},X_{F})$.\\
Let us first begin by bounding the density of the $(X_{I},X_{F})$.
\begin{lem}\label{boundDen}
Suppose that $\Vert f_{1}\Vert_{\infty}\leq A$ and let $\mathcal{S}$ be a Sawtooth model larger than $2$. Then there exist $K_{A}$ only depending on $A$ such that for all event $\mathcal{X}$ depending on $\lbrace X_{i},Y_{i}\rbrace_{i\geq 2}$ :
$$\Vert d_{X_{I}\vert \mathcal{X}}\Vert_{\infty}\leq K_{A}.$$
More precisely $K_{A}=4A^{2}$ fits. 
\end{lem}
This Lemma was already mentioned in the specific context of compositions in \cite{bender2004asymptotics}. We provide here a different proof.
\begin{proof}
By Lemma \ref{markov}, it suffices to prove it for a conditioning on $\lbrace X_{2}=x_{2}\rbrace$. From Lemma \ref{monden}, $d_{X_{I}\vert X_{2}=x_{2}}(x)$ is decreasing in $x$ and thus it is enough to bound $d_{X_{I}\vert X_{2}=x_{2}}(0)$. We have
$$d_{X_{I}\vert X_{2}=x_{2}}(0)=\frac{\int_{x_{2}}^{1}f_{1}(z)g_{1}(z-x_{2})dz}{\int_{x_{2}}^{1}F_{f_{1}}(z)g_{1}(z-x_{2})dz}\leq A\frac{\int_{x_{2}}^{1}g_{1}(z-x_{2})dz}{\int_{x_{2}}^{1}F_{f_{1}}(z)g_{1}(z-x_{2})dz}.$$
Remark that
$$\frac{\int_{x_{2}}^{1}g_{1}(z-x_{2})dz}{\int_{x_{2}}^{1}F_{f_{1}}(z)g_{1}(z-x_{2})dz}=\frac{1}{\mathbb{E}_{\tilde{Z}}(F_{f_{1}}(\tilde{Z}))},$$ with $\tilde{Z}$ being a random variable with density $\mathbf{1}_{z\geq x_{2}}g_{1}(z-x_{2})$. 
Since $\Vert F_{f_{1}}'\Vert\leq A$ and $F_{f_{1}}(1)=1$, $F_{f_{1}}(t)\geq 1/2$ on $[1-1/(2A)]$; moreover $z\mapsto g_{1}(z-x_{2})$ is increasing, thus $\mathbb{P}(\tilde{Z}\in[1-1/(2A),1])\geq \frac{1}{2A}$ and by Markov's inequality $\mathbb{E}_{\tilde{Z}}(F_{f_{1}}(\tilde{Z}))\geq 1/4A$.
Finally
$$d_{X_{I}\vert X_{2}=x_{2}}(0)\leq 4A^{2}.$$
\end{proof}
The next step is to get a bound on the first derivative of $d_{X_{I}}$. This is possible only if $g_{1}$ is also bounded by $A$ and the model is large enough.
\begin{lem}\label{boundDeriv}
Suppose that $\sup (\Vert f_{1}\Vert _{\infty},\Vert g_{1}\Vert_{\infty})\leq A$ and that $\mathcal{S}$ is a Sawtooth model with at least four particles. Then there exists a constant $R_{A}$ only depending on $A$ such that for any event $\mathcal{X}$ depending on $\lbrace X_{i+1},Y_{i}\rbrace_{i\geq 2}$,
$$\Vert (d_{X_{I}\vert \mathcal{X}})'\Vert_{\infty}\leq R_{A}.$$
\end{lem}
\begin{proof}
For exactly the same reasons as in the previous proof, it suffices to bound the derivative of the density conditioned on  $\mathcal{X}=\lbrace Y_{2}=y_{2}\rbrace$. 
The expression of the density probability yields
$$d_{X_{I}\vert Y_{2}=y_{2}}(x)=\frac{\int_{x}^{1}f_{1}(y_{1}-x)d_{Y_{1}\vert Y_{2}=y_{2}}(y_{1})dy_{1}}{\int_{0}^{1}\left(\int_{x}^{1}f_{1}(y_{1}-x)d_{Y_{1}\vert  Y_{2}=y_{2}}(y_{1})dy_{1}\right)dx}.$$
Let $\Delta=\int_{0}^{1}\left(\int_{x}^{1}f_{1}(y_{1}-x)d_{Y_{1}\vert Y_{2}=y_{2}}(y_{1})dy_{1}\right)dx$, which is independent of $x$. Then
\begin{align*}
\vert\frac{\partial}{\partial x}d_{X_{I}\vert Y_{2}=y_{2}}(x)\vert=&\frac{1}{\Delta}\vert \frac{\partial}{\partial x}\int_{x}^{1}f_{1}(y_{1}-x)d_{Y_{1}\vert  Y_{2}=y_{2}}(y_{1})dy_{1}\vert\\
=&\frac{1}{\Delta}\vert\int_{x}^{1}(\frac{\partial}{\partial x}f_{1}(y_{1}-x))d_{Y_{1}\vert  Y_{2}=y_{2}}(y_{1})dy_{1}-f_{1}(0)d_{Y_{1}\vert \mathcal{S}_{Y_{2}\leftarrow}, Y_{2}=y_{2}}(x)\vert\\
\leq &\frac{1}{\Delta}\left(\vert \int_{x}^{1}-(\frac{\partial}{\partial x}f_{1})(y_{1}-x)d_{Y_{1}\vert  Y_{2}=y_{2}}(y_{1})dy_{1}\vert +\vert f_{1}(0)d_{Y_{1}\vert Y_{2}=y_{2}}(x)\vert\right),
\end{align*}
Let us first bound the numerator. By the expression of the density of $Y_{1}$ conditioned on $Y_{2}=y_{2}$,
$$d_{Y_{1}\vert  Y_{2}=y_{2}}(y_{1})=\frac{F_{f_{1}}(y_{1})d_{Y_{1},\mathcal{S}_{Y_{1}\leftarrow}\vert Y_{2}=y_{2}}(y_{1})}{\mathbb{E}_{\tilde{Y}_{1}}(F_{f_{1}}(\tilde{Y}_{1}))},$$
with $\tilde{Y}_{1}$ having the density $d_{Y_{1},\mathcal{S}_{Y_{1}\leftarrow}\vert Y_{2}=y_{2}}$. Since $g_{1}$ is bounded by $A$, from Lemma \ref{boundDen}, $\vert d_{Y_{1},\mathcal{S}_{Y_{1}\leftarrow}\vert Y_{2}=y_{2}}\vert\leq K_{A}$. From Lemma \ref{monden}, $d_{Y_{1},\mathcal{S}_{Y_{1}\leftarrow}\vert Y_{2}=y_{2}}(y)$ is increasing in $y$, and $\vert F_{f_{1}}'\vert\leq A$, thus $\mathbb{E}_{\tilde{Y}_{1}}(F_{f_{1}}(\tilde{Y}_{1}))\geq \frac{1}{4A^{2}}$ and 
$$\vert f_{1}(0)d_{Y_{1}\vert \mathcal{S}_{Y_{2}\leftarrow}\vert Y_{2}=y_{2}}(x)\vert\leq 4A^{2}K_{A}^{2}.$$

Let us bound also the first term of the sum: $f_{1}$ being increasing, $\frac{\partial}{\partial x}f_{1}(y_{1}-x)\leq 0$ and we can thus remove the absolute value in this first term. An other application of Lemma \ref{boundDen} yields:
\begin{align*}
 \int_{x}^{1}-(\frac{\partial}{\partial x}f_{1})(y_{2}-x)d_{Y_{1}\vert  Y_{2}=y_{2}}(y_{1})dy_{1}&\leq K_{A}(\int_{x}^{1}(\frac{\partial}{\partial x}f_{1})(y_{2}-x)dy_{2})\\
 &\leq K_{A}((f_{1}(1-x)-f_{1}(0))\leq A\times K_{A}.\\
 \end{align*}
 The numerator is thus bounded by $AK_{A}+4A^{2}K_{A}^{2}$.\\
Interverting the integrals in $\Delta$ yields :
 $$\Delta=\int_{0}^{1}F_{f_{1}}(y_{1})d_{Y_{1}\vert Y_{2}=y_{2}}(y_{1})dy_{1}.$$
Since $F_{f_{1}}'$ is bounded by $A$ and $F_{f_{1}}(1)=1$, we can conclude as in the previous proof that $F_{f_{1}}(t)\geq \frac{1}{2A}$ on $[1-1/(2A),1]$. Moreover $Y_{1}$ is an upper particule and thus by Lemma \ref{monden}, $d_{Y_{1}\vert Y_{2}=y_{2}}(y_{1})$ is increasing in $y_{2}$. Since $\int_{[0,1]}d_{Y_{1}\vert Y_{2}=y_{2}}=1$, this implies that
$$\int_{1-1/(2A)}^{1}d_{Y_{1}\vert Y_{2}=y_{2}}(y_{1})dy_{1}\geq \frac{1}{2A},$$
and yields
 $\Delta\geq \frac{1}{4A^{2}}.$
 The bounds on the numerator and on $\Delta$ yield : 
 $$\vert \frac{\partial}{\partial x}d_{X_{I}\vert Y_{2}=y_{2}}(x)\vert\leq 4A^{3}(K_{A}+4AK_{A}^{2}).$$
\end{proof}
As an application of the latter Lemma, we can also prove that $y\mapsto F_{X_{I}\vert X_{F}=y}(t)$ is Lipchitz :
\begin{prop}\label{boundCumul}
Let $\mathcal{S}$ be a Sawtooth model with $n\geq 3$ lower particles. Suppose that $\lbrace f_{1},g_{1},f_{n},g_{n}\rbrace$ are bounded by $A>0$. Let $R_{A}$ be the constant of Lemma \ref{boundDeriv} (with $R_{A}\geq 1$). Then on a neighbourhood $[0,1/R_{A}]$ of $0$,
$$\mathcal{F}:\left\lbrace \begin{matrix} [0,1/R_{A}]&\rightarrow&(\mathcal{C}([0,1],\mathbb{R}),\Vert.\Vert)\\
y&\mapsto& F_{X_{I}\vert X_{F}=y}\end{matrix}\right.$$
 is Lipschitz with a Lipschitz constant $B_{A}$ only depending on $A$.
\end{prop}
\begin{proof}
It suffices to prove that for $x\in[0,1]$, $y\mapsto d_{X_{I}\vert X_{F}=y}(x)$ is Lipschitz on $[0,1/R_{A}]$ with a Lipschitz constant independent of $x$.\\
 From Lemma \ref{monden}, $d_{X_{F}}$ is decreasing and thus on $[1/R_{A},1]$, $d_{X_{F}}\leq d_{X_{F}}(1/R_{A})$. From Lemma \ref{boundDeriv}, $\vert\frac{\partial}{\partial y}d_{X_{F}}(y)\vert\leq R_{A}$ and thus on $[0,1/R_{A}]$, $d_{X_{F}}(y)\leq d_{X_{F}}(1/R_{A})+R_{A}(1/R_{A}-y)$. This implies that
\begin{align*} 
 \int_{[0,1]}d_{X_{F}}(y)dy\leq&\int_{0}^{1/R_{A}}d_{X_{F}}(1/R_{A})+R_{A}(1/R_{A}-y) dy+\int_{1/R_{A}}^{1}d_{X_{F}}(1/R_{A})\\
 \leq &d_{X_{F}}(1/R_{A})+\frac{1}{2R_{A}}.
 \end{align*}
 Since $\int_{[0,1]}d_{X_{F}}=1$, this implies that $d_{X_{F}}(1/R_{A})\geq 1-\frac{1}{2R_{A}}$, and thus that $d_{X_{F}}\geq 1-\frac{1}{2R_{A}}$ on $[0,1/R_{A}]$.\\
 From Lemma \ref{boundDeriv}, $\Vert \frac{\partial}{\partial y}d_{X_{F}\vert X_{I}=x}\Vert\leq R_{A}$. Thus since $\Vert f_{1}\Vert \leq A$ this yields by applying Lemma \ref{boundDen} on $d_{X_{I},X_{F}}(x,y)=d_{X_{F}\vert X_{I}=x}(y)d_{X_{I}}(x)$:
 $$\vert \frac{\partial}{\partial y}d_{X_{I},X_{F}}(x,y)\vert\leq  K_{A}R_{A}.$$
 Thus on $[0,1/R_{A}]$,
\begin{align*} 
 \vert \frac{\partial}{\partial y}d_{X_{I}\vert X_{F}=y}(x)\vert =&\frac{1}{d_{X_{F}}(y)}\vert \frac{\partial}{\partial y}d_{X_{I},X_{F}}(x,y)-\frac{d_{X_{I},X_{F}}(x,y)\frac{\partial}{\partial y}d_{X_{F}}(y)}{d_{X_{F}}(y)}\vert\\
 \leq& \frac{1}{1-1/(2R_{A})}(K_{A}R_{A}+\frac{R_{A}K_{A}^{2}}{1-1/(2R_{A})}).
 \end{align*}
Set $B_{A}=\frac{1}{1-1/(2R_{A})}(K_{A}R_{A}+\frac{R_{A}K_{A}^{2}}{1-1/(2R_{A})})$. Then $\mathcal{F}$ is $B_{A}-$Lipschitz on $[0,1/R_{A}]$.
 \end{proof}
\subsection{Behavior of $\lbrace X_{i}\rbrace$ for large models}
The purpose of this subsection is to find for a model $\mathcal{S}$ a large set of intermediate particles $\lbrace X_{r}\rbrace$ for which almost surely one of these particles is close to $0$ and such that $F_{X_{I}\vert X_{r}=0}$ is essentially the same for all particles of this set. 
The first part is a essentially propability computation :
\begin{prop}\label{positParticle}
Let $\eta>0,\epsilon>0$. There exists $N_{0}$ such that for any model $\mathcal{S}$ of size $N$ larger than $N_{0}+4$ and for any $2\leq r\leq N-N_{0}$, $y_{r+N_{0}}\in[0,1]$,
$$\mathbb{P}(\bigcup_{r\leq i\leq r+N_{0}}\lbrace X_{i}<\eta\rbrace\vert Y_{r+N_{0}}=y_{r+N_{0}})\geq 1-\epsilon.$$
\end{prop}
\begin{proof}
Let $N_{0}$ be an integer to specify later and $\mathcal{S},r$ as in the statement of the Proposition. Let $\tilde{P}=\mathbb{P}(\bigcap_{r\leq i\leq r+N_{0}}\lbrace X_{i}\geq\eta\rbrace\vert Y_{r+N_{0}}=y_{r+N_{0}})$. Condition this probability on the value of $Y_{r-1}=y_{r-1}$ and denote by $P$ this quantity. Then we have 
$$P=\frac{\int_{[\eta,1]^{N_{0}+1}}\int_{[\eta,1]^{N_{0}}}\prod_{r\leq i\leq r+N_{0}} \mathbf{1}_{x_{i}\leq y_{i},y_{i-1}}f_{i}(y_{i}-x_{i})g_{i}(y_{i-1}-x_{i})\prod dx_{i} dy_{i}}{\int_{[0,1]^{N_{0}+1}}\int_{[0,1]^{N_{0}}}\prod_{r\leq i\leq r+N_{0}} \mathbf{1}_{x_{i}\leq y_{i},y_{i-1}}f_{i}(y_{i}-x_{i})g_{i}(y_{i-1}-x_{i})\prod dx_{i}\prod dy_{i}}.$$
We can operate the linear change of variable 
$$\left\lbrace\begin{matrix}
x_{i}\rightarrow u_{i}=\frac{1}{1-\eta} (x_{i}-\eta)\\
y_{i}\rightarrow v_{i}=\frac{1}{1-\eta} (y_{i}-\eta)
\end{matrix}\right.$$
on the numerator. This yields 
\begin{align*}
P=&\int\frac{(1-\eta)^{2N_{0}+1}\prod_{r\leq i\leq r+N_{0}} \mathbf{1}_{v_{i}\geq u_{i},u_{i+1}}f_{i}((1-\eta)(v_{i}-u_{i}))g_{i}((1-\eta)(v_{i}-u_{i+1}))}{\int_{[0,1]^{2N_{0}+1}}\prod_{r\leq i\leq r+N_{0}} \mathbf{1}_{x_{i}\leq y_{i},y_{i-1}}f_{i}(y_{i}-x_{i})g_{i}(y_{i-1}-x_{i})\prod dx_{i}\prod dy_{i}}\\
&\times g_{r}(y_{r-1}-((1-\eta)u_{r}+\eta))f_{r+N_{0}}(y_{r+N_{0}}-((1-\eta)u_{r+N_{0}}+\eta))\prod du_{i} dv_{i}.
\end{align*}
Now recall that each $f_{i}, g_{i}$ is increasing. Moreover $(1-\eta)(u-v)\leq (u-v)$, and $y-((1-\eta)u+\eta)\leq y-u$. Thus 
$$P\leq (1-\eta)^{2N_{0}+1}.$$
Let $N_{0}$ be such that $ (1-\eta)^{2N_{0}+1}\leq \epsilon$. Then by averaging on $y_{r-1}$,
$$\tilde{P}\leq \epsilon,$$
and this concludes the proof.
\end{proof}
As said before, it is also necessary that $F_{X_{I}\vert X_{r}=0}$ remains almost constant among this subset of particles. This is possible for large Sawtooth models, thank to the monotony results of Proposition \ref{monoton2} :
\begin{prop}\label{dependSplitPart}
Let $A,\epsilon>0$, $M\in\mathbb{N}^{*}$. There exists $N_{\epsilon,A,M}$ such that for any Sawtooth model bounded by $A$ and of size $N\geq N_{\epsilon,A,M}$, there exists $1\leq r\leq N-M$ such that for $r\leq i,j\leq r+M$,
$$\Vert F_{X_{I}\vert X_{i}=0}-F_{X_{I}\vert X_{j}=0}\Vert\leq \epsilon.$$
\end{prop}
\begin{proof}
Let $\mathcal{S}$ be a Sawtooth model bounded by $A$ and of size $N$.\\
Denote by $F_{i}$ the function $t\mapsto F_{X_{I}\vert X_{i}=0}(t)$ for $2\leq i\leq N$. By Lemma \ref{boundDen}, all the $F_{i}$ are $K_{A}-$Lipschitz. Let $K=\lfloor \frac{2K_{A}}{\epsilon}\rfloor$. It suffices to find $r\geq 2$ such that for all $r\leq i,j\leq r+M$, and all $0\leq k\leq K$,
$$\vert F_{i}(\frac{k}{K})-F_{j}(\frac{k}{K})\vert\leq \frac{\epsilon}{3}.$$
Denote by $v_{i}\in[0,1]^{K+1}$ the vector $(F_{i}(\frac{k}{K}))_{0\leq k\leq K}$. Let $N_{\epsilon,A,M}=M(\lfloor\frac{3}{\epsilon}\rfloor+1)^{K+1}$. Then if $N\geq N_{\epsilon,A,M}$, by the Dirichlet principle on $[0,1]^{K+1}$, there exists an hypercube of size $\frac{\epsilon}{3}$ that contains at least $M$ distinct points $v_{i_{1}},\dots,v_{i_{M}}$ (with $i_{1}<\dots<i_{M}$). Moreover by Proposition \ref{monoton2}, for all fixed $0\leq k\leq K$, $v_{i}(k)=F_{i}(\frac{k}{K})$ is decreasing and thus for all $i_{1}\leq i\leq i_{M}$, $v_{i_{1}}(k)\leq v_{i}(k)\leq v_{i_{M}}(k)$. This yields for all $i_{1}\leq i,j\leq i_{M}$, 
$$\Vert v_{i}-v_{j}\Vert_{\infty}\leq \frac{\epsilon}{3}.$$ 
\end{proof}
\subsection{Proof of Theorem \ref{indepPart}}
Theorem \ref{indepPart} is a consequence of the following proposition :
\begin{prop}\label{mainStep}
Let $A>0$. For all $\epsilon>0$, there exists a number $N_{A,\epsilon}\geq 0$ such that for all Sawtooth model $\mathcal{S}$ bounded by $A$ and with $2n\geq N_{A,\epsilon}$ particles we have :
$$\vert F_{X_{I}\vert X_{F}=y}(t)-F_{X_{I}}(t)\vert\leq \epsilon.$$
for all $t,y\in[0,1]$.
\end{prop}
\begin{proof}
Set $\eta=\inf(\frac{1}{R_{A}},\frac{\epsilon}{B_{A}})$ with $R_{A},B_{A}$ the constants given respectively by Lemma \ref{boundDeriv} and Proposition \ref{boundCumul}. Let $N_{0}$ be the constant given for $\eta$ and $\epsilon$ by Proposition \ref{positParticle}. And finally set $N_{A,\epsilon}=N_{\epsilon/4,A,N_{0}}+4$ given by Proposition \ref{dependSplitPart}.\\
Let $\mathcal{S}$ be a Sawtooth model bounded by $A$ of size larger than $N_{A,\epsilon}$. Then by Proposition \ref{dependSplitPart}, there exists $2\leq r\leq N_{A,\epsilon}-2-N_{0}$ such that for all $r\leq i,j\leq r+N_{0}$,
$$\Vert F_{X_{I}\vert X_{i}=0}-F_{X_{I}\vert X_{j}=0}\Vert_{\infty}\leq \epsilon.$$
Denote $t=r+N_{0}$ and let $y_{t}\in[0,1]$.
For $r\leq i\leq r+N_{0}$, set $L_{i}=\lbrace X_{i}\leq \eta \cap \lbrace \forall s>i, X_{s}>\eta\rbrace\rbrace $. Note that $L_{i}\cap L_{j}=\emptyset$ for all $i\not = j$ and $\bigcup L_{i}=L$ with $L=\bigcup_{r\leq i\leq r+N_{0}}\lbrace X_{i}\leq\eta\rbrace$. Moreover since $L_{i}$ is $(X_{s},Y_{s})_{s\geq i}-$measurable, by Lemma \ref{markov}, conditioning $X_{I}$ on $\lbrace X_{i}=u,Y_{t}=y_{t}\rbrace\cap L_{i}$ is the same as conditioning $X_{I}$ on $\lbrace X_{i}=u\rbrace$. Thus
\begin{align*}
\Vert F_{X_{I}\vert L_{i},Y_{t}=y_{t}}-F_{X_{I}\vert X_{r}=0}\Vert_{\infty}=&\Vert \int_{0}^{\eta}(F_{X_{I}\vert X_{i}=u}-F_{X_{I}\vert X_{r}=0})d_{X_{i}\vert L_{i} ,Y_{t}=y_{t}}(u)du\Vert_{\infty}\\
\leq &\int_{0}^{\eta}\Vert F_{X_{I}\vert X_{i}=u}-F_{X_{I}\vert X_{r}=0}\Vert_{\infty}d_{X_{i}\vert L_{i},Y_{t}=y_{t}}(u)du\\
\leq &2\epsilon,
\end{align*}
by the choice of $\eta$. Recall that if $A=\bigcup A_{i}$, with $A_{i}$ disjoint events, then for any event $C$,
$$\mathbb{P}(C\vert A)=\sum \mathbb{P}(C\vert A_{i})\mathbb{P}(A_{i}\vert A)$$
In particular for $L=\bigcup_{i}L_{i}$ this yields
\begin{align*}
\Vert F_{X_{I}\vert L,Y_{t}=y_{t}}-F_{X_{I}\vert X_{r}=0}\Vert=&\Vert \sum_{i} (F_{X_{I}\vert L_{i},Y_{t}=y_{t}}-F_{X_{I}\vert X_{r}=0})\mathbb{P}(L_{i}\vert L,Y_{t}=y_{t})\Vert_{\infty} \\
\leq& \sum_{i}\Vert (F_{X_{I}\vert L_{i},Y_{t}=y_{t}}-F_{X_{I}\vert X_{r}=0}\Vert_{\infty}\mathbb{P}(L_{i}\vert L, Y_{t}=y_{t})\\
\leq& 2\epsilon.
\end{align*}
By Proposition \ref{positParticle} and the choice of $N_{0}$, $\mathbb{P}(L\vert Y_{t}=y_{t})\geq 1-\epsilon$, and thus
$$\Vert F_{X_{I}\vert Y_{t}=y_{t}}-F_{X_{I}\vert X_{r}=0}\Vert_{\infty}\leq 3\epsilon.$$
By averaging on $y_{t}$ with the density $d_{Y_{t}\vert X_{F}=y}$ we get 
$$\Vert F_{X_{I}\vert X_{F}=y}-F_{X_{I}}\Vert_{\infty}\leq 4\epsilon.$$
\end{proof}
Let us end the proof of the Theorem \ref{indepPart}, which consists essentially in a rewriting in terms of densities of the latter Proposition.
\begin{proof}
Let $A>0,\epsilon>0$. Set $\epsilon_{1}=\frac{(\epsilon/K_{A)}^{2}}{2R_{A}}$ and let $\mathcal{S}$ be a Sawtooth model bounded by $A$ of size larger than $N_{A,\epsilon_{1}}$ ($N_{A,\epsilon_{1}}$ being given by Proposition \ref{mainStep}). Then from Proposition \ref{mainStep}, for $y\in [0,1]$,
\begin{equation}\label{eq 1}
\Vert F_{X_{I}\vert X_{F}=y}-F_{X_{i}}\Vert_{\infty} \leq \frac{(\epsilon/K_{A})^{2}}{2R_{A}}.
\end{equation}
Moreover the following result holds for $C^{1}-$functions on $[0,1]$:
\begin{lem}\label{dens/Cumul}
Let $f,g:[0,1]\rightarrow [0,1]$ be two $C^{1}-$ functions, such that $\Vert f'\Vert_{\infty},\Vert g'\Vert_{\infty}\leq M$. Then for $\epsilon>0$, if $F,G$ are two primitive of $f,g$ and 
$$\Vert F-G\Vert_{\infty}\leq\frac{\epsilon^{2}}{2M},$$
then $\Vert f-g\Vert_{\infty}\leq \epsilon$.
\end{lem}
Applying this Lemma to \eqref{eq 1} yields for $y\in[0,1]$,
$$\Vert d_{X_{I}\vert X_{F}=y}-d_{X_{i}}\Vert_{\infty} \leq\epsilon/K_{A}.$$
And finally,
$$\vert d_{X_{I},X_{F}}(x,y)-d_{X_{I}}(x)d_{X_{F}}(y)\vert=\vert d_{X_{F}}(y)\Vert \vert d_{X_{I}\vert X_{F}=y}(x)-d_{X_{I}}(x))\vert\leq K_{A}\frac{\epsilon}{K_{A}}\leq \epsilon.$$
\end{proof}
\section{Application to compositions}
Theorem \ref{indepPart} can be applied to the framework of compositions :
\begin{cor}\label{indepcom}
Let $A\geq 0,\epsilon>0$. There exists $n\geq 0$ such that for any composition $\lambda$ of size larger than $n$ with every runs bounded by $A$,
$$\Vert d_{\mathcal{S}_{\lambda}}(x,y)-d_{\mathcal{S}_{\lambda}}(x)d_{\mathcal{S}_{\lambda}}(y)\Vert<\epsilon$$
\end{cor}
\begin{proof}
Each run of $\lambda$ of length $l$ yields a density function $\gamma_{l}$ in $\mathcal{S}_{\lambda}$, and $\Vert \gamma_{l}\Vert_{\infty}=l-1$. Thus if any run of $\lambda$ is bounded by $A$, then all the density functions $\lbrace f_{i},g_{i}\rbrace$ in $\mathcal{S}_{\lambda}$ are bounded by $A-1$. It suffices then to apply Theorem \ref{indepPart}.
\end{proof}
The purpose of this section is to strenghen Corollary \ref{indepcom} and to prove the following Theorem :
\begin{thm}\label{indepGen}
Let $\epsilon>0$, $A\geq 0$. There exists $n\geq 0$ such that for any composition $\lambda$ of size larger than $n$ with first and last run bounded by $A$,
\begin{equation}\label{result}
\Vert d_{\mathcal{S}_{\lambda}}(x,y)-d_{\mathcal{S}_{\lambda}}(x)d_{\mathcal{S}_{\lambda}}(y)\Vert<\epsilon.
\end{equation}
\end{thm}
This Theorem was actually Conjecture 4 in \cite{bender2004asymptotics}. By Lemma \ref{NNRPdes}, the latter Theorem is equivalent to Theorem \ref{goal1}.
The proof of Theorem \ref{indepGen} is followed by some applications.
\subsection{Effect of a large run on the law of $(X_{I},X_{F})$}
From Corollary \ref{indepcom}, it is enough to prove that the presence of a large run inside the composition disconnects the behaviors of $X_{I}$ and $X_{F}$. The main reason for this is the Lemma below: for each composition $\lambda$, denote by $\lambda^{+}$ the composition $\lambda$ with a cell added on the last run, and by $\lambda^{-}$ the composition $\lambda$ with a cell removed on the last run.
\begin{lem}
Let $A>0$ and $\lambda$ a composition with more than three runs and with the first run smaller than $A$. If the last run of $\lambda$ is of size $R$,
$$\Vert d_{\mathcal{S}_{\lambda}}-d_{\mathcal{S}_{\lambda^{+}}}\Vert_{\infty}\leq \frac{K_{A}}{R-1},$$
\end{lem}
where $K_{A}$ is the bound on the density of $X_{I}$ as defined in Lemma \ref{boundDen}.
\begin{proof}
Let us prove it in the case where the first run of $\lambda$ is increasing and the last run decreasing, the other cases having the same proofs. The expression \eqref{eqnLoi} yields
$$d_{X_{I},X_{F}}^{\mathcal{S}_{\lambda^{+}}}(x,y)=\frac{\int_{y}^{1}d_{X_{I},X_{F}}^{\mathcal{S}_{\lambda}}(x,z)dz}{\int_{[0,1]^{2}} (\int_{y}^{1}d_{X_{I},X_{F}}^{\mathcal{S}_{\lambda}}(x,z))dxdy}.$$
Thus by integrating on $y$ and then interverting the integrals, this yields
\begin{align*}
d_{X_{I}}^{\mathcal{S}_{\lambda^{+}}}(x)=&\frac{\int_{0}^{1}(\int_{0}^{1}d_{X_{I},X_{F}}^{\mathcal{S}_{\lambda}}(x,z)\mathbf{1}_{y\leq z}dy)dz}{\int_{[0,1]^{2}}(\int_{0}^{1}d_{X_{I},X_{F}}^{\mathcal{S}_{\lambda}}(x,z)\mathbf{1}_{y\leq z}dy)dxdz}\\
=&\frac{\int_{0}^{1}d_{X_{I},X_{F}}^{\mathcal{S}_{\lambda}}(x,z)zdz}{\int_{[0,1]^{2}}d_{X_{I},X_{F}}^{\mathcal{S}_{\lambda}}(x,z)zdzdx}.
\end{align*}
Factorizing by $d_{X_{I}}^{\mathcal{S}_{\lambda}}(x)$ makes a conditional expectation appear and thus
$$d_{X_{I}}^{\mathcal{S}_{\lambda^{+}}}(x)=d_{X_{I}}^{\mathcal{S}_{\lambda}}(x)\frac{\mathbb{E}_{\mathcal{S}_{\lambda}}(X_{F}\vert X_{I}=x)}{\mathbb{E}_{\mathcal{S}_{\lambda}}(X_{F})}.$$
Moreover Proposition \ref{domNNRPinit} yields
$$F_{Z_{1}}\leq F_{X_{F}\vert X_{I}=x}\leq F_{Z_{2}},$$
with $F_{Z_{1}}=\Gamma^{-}(F_{\gamma_{R}})$ and $F_{Z_{2}}= \Gamma^{-}(\gamma_{R})$. Since $\Gamma^{-}(F_{\gamma_{R}})(t)=1-(1-t)^{R}$ and $\Gamma^{-}(\gamma_{R})(t)=1-(1-t)^{R-1}$, by stochastic dominance applying Proposition \ref{definitionStochastDom} gives
$$\frac{1}{R}\leq\mathbb{E}_{\mathcal{S}_{\lambda}}(X_{F}\vert X_{I}=x)\leq \frac{1}{R-1}.$$
Integrating the latter result on $x$ yields $\frac{1}{R}\leq \mathbb{E}_{\mathcal{S}_{\lambda}}(X_{F})\leq \frac{1}{R-1}$, and thus
$$\frac{R-1}{R}\leq \frac{\mathbb{E}_{\mathcal{S}_{\lambda}}(X_{F}\vert X_{I}=x)}{\mathbb{E}_{\mathcal{S}_{\lambda}}(X_{F})}\leq\frac{R}{R-1}.$$
This yields
$$\vert d_{\mathcal{S}_{\lambda^{+}}}(x)-d_{\mathcal{S}_{\lambda}}(x)\vert\leq\vert d_{\mathcal{S}_{\lambda}}(x)\vert\frac{1}{R-1}\leq \frac{K_{A}}{R-1}.$$
\end{proof}
In particular the previous Lemma can used to bound the conditional law of the first particle with respect to the last one. For each composition $\lambda$, and each cell $i\in\lambda$, denote by $\lambda_{\rightarrow i}$ the composition $\lambda$ truncated just after the cell $i$. Moreover denote by $R_{int}(\lambda)$ the set of all runs of $\lambda$ except the first and last ones.
\begin{prop}\label{indeplargerun}
Let $A\geq 0$ and $\lambda$ a composition with first run bounded by $A$. Then
$$\Vert F_{X_{I}\vert X_{F}=x}-F_{X_{I}}\Vert_{\infty} \leq\frac{K_{A}}{\sup_{s\in R_{int}(\lambda)}l(s)-2}.$$
\end{prop}
\begin{proof}
Let $t\in[0,1]$. Let $s_{0}$ be the run with maximal length $R$ in $R_{int}$ and let $i_{0}$ be the rightest cell of this run.  This cell corresponds to a particle $X_{i}$ of $Y_{i}$ in $\mathcal{S}_{\lambda}$. Let us assume without loss of generality that this particle is a lower one. From Proposition \ref{monoton}, $F_{X_{1}\vert X_{r}=x}(t)$ is decreasing in $x$ and thus
\begin{align*}
\vert F_{X_{I}\vert X_{F}=x}(t)-F_{X_{I}}(t)\vert=&\vert  F_{X_{I}\vert X_{F}=x}(t)-\int_{X_{F}}F_{X_{I}\vert X_{F}=x}(t)d_{X_{F}}(x)dx\vert\\
\leq &\vert F_{X_{I}\vert X_{F}=0}(t) -F_{X_{I}\vert X_{F}=1}(t)\vert\\
\leq & F_{X_{I}\vert X_{F}=0}(t) -F_{X_{I}\vert Y_{n}=1}(t).
\end{align*}
Moreover from  Proposition \ref{monoton} and Proposition \ref{monoton2},
$$F_{X_{I}\vert X_{F}=0}(t)\leq F_{X_{I}\vert \mathcal{S}_{\lambda}\rightarrow Y_{n}}(t)\leq F_{X_{I}\vert \mathcal{S}_{\lambda}\rightarrow Y_{i}}(t)\leq F_{X_{I}\vert X_{i}=0},$$
and 
$$F_{X_{I}\vert Y_{n}=1}(t)\geq  F_{X_{I}\vert \mathcal{S}_{\lambda}\rightarrow X_{n}}(t)\geq F_{X_{I}\vert \mathcal{S}_{\lambda}\rightarrow X_{i}}(t).$$
These inequalities imply
$$\vert F_{X_{I}\vert X_{F}=x}(t)-F_{X_{I}}(t)\vert\leq F_{X_{I}\vert X_{i}=0}(t)-F_{X_{I}\vert \mathcal{S}_{\lambda}\rightarrow X_{i}}(t).$$
From the expression \eqref{eqnLoi}, $F_{X_{I}\vert \mathcal{S}_{\lambda}\rightarrow X_{i}}(t)=F_{X_{I}, \mathcal{S}_{\lambda_{\rightarrow i_{0}}}}(t)$ and $F_{X_{I}\vert X_{i}=0}(t)=F_{X_{1}, \mathcal{S}_{\lambda_{\rightarrow i_{0}}^{-}}}(t)$. Thus with the previous Lemma, since the last run of $\lambda_{\rightarrow i_{0}}^{-}$ is of size $R-1$,
$$\vert F_{X_{I}\vert X_{F}=x}(t)-F_{X_{I}}(t)\vert\leq \vert F_{X_{I}, \mathcal{S}_{\lambda_{\rightarrow i_{0}}}}(t)-F_{X_{I}, \mathcal{S}_{\lambda_{\rightarrow i_{0}}^{-}}}(t)\vert\leq \frac{K_{A}}{R-2}.$$
\end{proof}
\subsection{Proof of Theorem \ref{indepGen}}
The latter Proposition together with Lemma \ref{dens/Cumul} yields Theorem \ref{indepGen} in case $d_{X_{I}}'$ remains bounded. However the bound of the derivative in Lemma \ref{boundDeriv} requires also a bound on the second run, and the latter is not assume in our case. We should thus deal with this case before getting the general proof. Let us first consider a particular case.
\begin{lem}\label{particularCase}
Let $\lambda_{b}$ be the composition with three runs of respective length  $2$, $b$ and $2$, and $d_{b}(x,y)=d_{X_{I},\vert Y_{2}=y}(x)$. Then the following convergence holds: 
$$\lim\limits_{b\rightarrow \infty}\sup_{[0,1]^{2}}(d_{b}(x,y)-(1-x^{b}))= 0.$$
In particular the asymptotic independence :
\begin{equation}\label{particularCompo}
\lim\limits_{b\rightarrow\infty}\sup_{x,y,y'}(d_{b}(x,y)-d_{b}(x,y'))=0.
\end{equation}
is valid.
\end{lem}
\begin{proof}
After integrating in \eqref{eqnLoi} the coordinates of the particles inside the composition :
\begin{equation}\label{eqnDen}
d_{b}(x,y)=\frac{1-x^{b}-(1-y)^{b}+((x-y)\wedge 0)^{b}}{(1-1/(b+1))(1-(1-y)^{b})+y/(b+1)(1-y)^{b}}.
\end{equation}
Let us show that $\lim\limits_{b\rightarrow\infty}d_{b}(x,y)-(1-x^{b+1})=0$ uniformly in $x$ and $y$. In the denominator of \eqref{eqnDen}, letting $b$ go to $+\infty$ yields
$$(1-\frac{1}{b+1})(1-(1-y)^{b})+y/(b+1)(1-y)^{b}\sim_{b\rightarrow \infty} 1-(1-y)^{b},$$
with the equivalent being uniform in $x$ and $y$. Indeed $$\frac{y/(b+1)(1-y)^{b}}{1-(1-y)^{b}}=\frac{1}{b+1}\frac{(1-y)^{b}}{\sum_{k=0}^{b-1}(1-y)^{k} }\leq \frac{1}{b+1}.$$
Since for $x\in [0,1/2], y\in [1/2,1]$, $d_{b}(x,y)$ converges uniformly to $1$, it suffices to consider in the sequel that $x\in [1/2,1]$ and $y\in [0,1/2]$. Let $\Delta$ be defined as
\begin{align*}
\Delta(x,y)=&\frac{1-x^{b}-(1-y)^{b}+(x-y)^{b}}{1-(1-y)^{b}}-(1-x^{b})\\
=&(1-\frac{x^{b}-(x-y)^{b}}{1-(1-y)^{b}})-(1-x^{b})=\frac{(x-y)^{b}-(1-y)^{b}x^{b}}{1-(1-y)^{b}}.
\end{align*}
A derivative computation shows that $\Delta(x,y)\leq \frac{1}{b}$, which proves the uniform convergence.
Since $\lim\limits_{b\rightarrow\infty}\Vert d_{b}(x,y)-(1-x^{b+1})\Vert_{\infty,[0,1]^{2}}=0$, 
$$\lim\limits_{b\rightarrow\infty}\sup_{y,y',x} (d_{b}(x,y)-d_{b}(x,y'))=0.$$
\end{proof}
From the latter result can be deduced the asymptotic independence with a large second run :
\begin{lem}\label{secondRun}
Let $A,\epsilon>0$. There exist $B_{A}\in\mathbb{N}$ such that if $\lambda$ is a composition with at least three runs, the extreme runs bounded by $A$ and the second run larger than $B_{A}$, then
$$\Vert d_{X_{I},X_{F}}-d_{X_{I}}d_{X_{F}}\Vert_{\infty}\leq \epsilon$$
\end{lem}
\begin{proof}
Let $\lambda$ be a composition with first run of length $a$ and second run of length $b$. From the definition of the density $d_{X_{I},X_{F}}$ in \eqref{eqnLoi}, conditioning the law of $X_{I}$ on the position $x_{P}$ of the particle $P=a+b$ yields
$$d_{X_{I}\vert x_{p}=y}(x)=\frac{\int_{x}^{1}(\int_{0}^{z_{1}\wedge y}(z_{1}-x)^{a-2}(z_{1}-z_{2})^{b-2}dz_{2})dz_{1}}{\mathcal{Z}}.$$
Let $2\leq a\leq A$. Then 
$$d_{X_{I}\vert x_{p}=y}(x)=\frac{\int_{x}^{1}(u-x)^{a-3}d_{b}(u,y)du}{\frac{1}{a-2}\int_{0}^{1}u^{a-2}d_{b}(u,y)du}$$
From the first part of Lemma \ref{particularCase}, $\vert d_{b}(u,y)-(1-u^{b})\vert\to_{b\rightarrow\infty} 0$ uniformly in $u$ and $y$, and thus
$$\frac{1}{a-2}\int_{0}^{1}u^{a-2}d_{b}(u,y)du\rightarrow_{b\rightarrow \infty}\frac{1}{(a-2)(a-1)},$$
uniformly in $y$. Since $a$ is bounded by $A$, and from the second part of Lemma \ref{particularCase},
$$\Vert d_{X_{I}\vert x_{p}=y}-d_{X_{I}\vert x_{p}=y'}\Vert_{\infty}\leq A^{2}\sup_{y,y',x} (d_{b}(x,y)-d_{b}(x,y'))\rightarrow 0$$
uniformly in $y$.
Thus for $b$ large enough, $\Vert d_{X_{I}\vert x_{p}=y}-d_{X_{I}\vert x_{p}=y'}\Vert<\epsilon/A$ for all $y,y'$; then averaging on the law of $x_{p}$ conditioned on $X_{F}=y$ yields $\vert d_{X_{I}\vert X_{F}=y}-d_{X_{I}\vert X_{F}=y'}\vert<\epsilon/A$ for all $y,y'$. And finally this implies that 
$$\Vert d_{X_{I},X_{F}}-d_{X_{I}}d_{X_{F}}\Vert_{\infty}\leq \epsilon.$$\\
\end{proof}

The proof of Theorem \ref{indepGen} is just a gathering of all the previous results :

\begin{proof}
Let $A,\epsilon>0$. Since the first and last runs are bounded by $A$, any composition large enough has at least three runs. Let $B_{A}$ be given by Lemma \ref{secondRun}, $R$ be the associate constante given by Lemma \ref{boundDeriv} for $B_{A}$, and set $C=\frac{2K_{A}R}{(\epsilon/A)^{2}}$. Finally, let $n$ be the integer given by Corollary \ref{indepcom} for compositions of runs bounded by $C$. Suppose that $\lambda$ is a composition larger than $n$. By Lemma \ref{secondRun}, if the second run is larger than $B_{A}$, \eqref{result} is verified. Thus we can suppose that the second run is bounded by $B_{A}$. \\
 If $\lambda$ has a run larger than $C$, then from Proposition \ref{indeplargerun}, 
$$\Vert F_{X_{I}\vert X_{F}=x}-F_{X_{I}}\Vert_{\infty}\leq\frac{K_{A}}{C-1}\leq \frac{(\epsilon/A)^{2}}{2R}.$$
But from Lemma \ref{boundDeriv}, $d_{X_{I}}'$ is bounded by $R$, thus the latter inequality yields with Lemma \ref{dens/Cumul} :
$$\Vert d_{X_{I}\vert X_{F}=y}-d_{X_{I}}\Vert\leq \epsilon/A.$$
And $d_{X_{I}}$ being bounded by $A$, this yields \eqref{result}.\\
Thus we can assume that all the runs of $\lambda$ are bounded by $C$. Once again by the choice of $n$ and Corollary \ref{indepcom}, \eqref{result} is verified.
\end{proof}
Note that we actually proved something stronger than Theorem \ref{indepGen}, namely :
\begin{cor}\label{firstRun}
Let $A,\epsilon>0$. There exists $n$ such that for every composition $\lambda$ of size larger than $n$ and first run bounded by $A$, and for all $y,y'\in[0,1]$,
$$\Vert d_{X_{I}\vert X_{F}=y}-d_{X_{I}\vert X_{F}=y'}\Vert \leq \epsilon.$$
\end{cor}
\subsection{Consequences and proof of Theorem \ref{goal2}}
Here are some interesting consequences of Theorem \ref{indepGen}. Let us first remove the constraints on the extreme runs. 
\begin{lem}\label{cumulDist}
Let $\epsilon>0$. There exists $n\geq 0$ such that for all compositions larger than $n$ with at least two runs, 
$$\sup_{(y,y')\in [0,1]^{2}}(\Vert F_{X_{I}\vert X_{F}=y}-F_{X_{I}\vert X_{F}=y'}\Vert_{\infty} )\leq \epsilon.$$
\end{lem}
\begin{proof}
Let $R$ be the length of the first run of a composition $\lambda$. From Proposition \ref{domNNRPinit} applied to $\mathcal{S}_{\lambda}$, 
$$1-(1-t)^{R}\leq F_{X_{I}\vert X_{F}=y}(t)\leq 1-(1-t)^{R-1}.$$
Since $\sup_{[0,1]}(u^{R-1}-u^{R})\rightarrow_{R\rightarrow \infty}0$, there exists $A$ such that for any composition with first run larger than $A$,
$$\sup_{[0,1]^{2}}\Vert F_{X_{I}\vert X_{F}=y}-F_{X_{I}\vert X_{F}=y'}\Vert _{\infty}\leq \epsilon.$$
Applying Corollary \ref{firstRun} to $A,\epsilon$ yields that there exists $n$ such that for any composition larger than $n$,
$$\sup_{[0,1]^{2}}\Vert F_{X_{I}\vert X_{F}=y}-F_{X_{I}\vert X_{F}=y'}\Vert_{\infty} \leq \epsilon.$$
\end{proof}
This result can be adapted to show that the law of the first particle depends only on the neighbouring particles : for any composition $\lambda$ of size $N$, and $n\leq N$, denote by $\lambda(n)$ the composition $\lambda$ containing only the $n$ first cells. 
\begin{prop}
Let $\epsilon>0$. There exists $n_{0}\geq 1$ such that for any $n\geq n_{0}$ and any composition $\lambda$ of size larger than $n$ with first run smaller than $n$,
$$\Vert F_{X_{I}}^{\mathcal{S}_{\lambda}}-F_{X_{I}}^{\mathcal{S}_{\lambda(n)}}\Vert_{\infty}\leq \epsilon.$$
\end{prop}
The proof consists only in an averaging of the inequality of the previous Lemma.\\
We will close this paper by proving Theorem \ref{goal2} :
\begin{proof}
By iteration it is enough to prove the result in the case $r=2$. Let $1\leq i_{1}<i_{2}\leq N$. Denote by $\nu_{1}$ (resp. $\nu_{2}$, resp. $ \nu_{3}$) the composition containing all the cells $i$ of $\lambda$ such that $i\leq i_{1}$ (resp. $i_{1}\leq i\leq i_{2}$, resp $i_{2}\leq i$). \\
Let $a$ be the length of the last run of $\nu_{1}$ and $b$ the length of the first run of $\nu_{2}$. We suppose that these two runs are increasing (the proof is the same in other cases). From \eqref{eqnLoi}, conditioning on the position $x_{i_{1}-a}=z$ and $x_{i_{1}+b}=z'$ yields the density of $x_{i_{1}}$
$$\tilde{d}_{z,z'}(x)=\frac{(\int_{0}^{z\wedge x}(x-z)^{a}dz)(\int_{z\vee x}^{1}(z'-x)^{b}dz')}{\int_{0}^{1}(\int_{0}^{z\wedge x}(x-z)^{a}dz)(\int_{z\vee x}^{1}(z'-x)^{b}dz')dx}.$$
From the latter expression, as $\min(a,b)\rightarrow +\infty$, $d_{\pi}(\mu(x_{i_{1}}),\delta_{\frac{a}{a+b}})$ goes to zero, implying the indepedence. Thus we assume that $a$ and $b$ are bounded by some constant $R$ and that the same holds for the first run of $\nu_{3}$ and the last run of $\nu_{2}$.\\
From \eqref{eqnLoi}, the law of $x_{i_{1}}$ and $x_{i_{2}}$ is 
$$d_{x_{i_{1}},x_{i_{2}}}(x,y)=\frac{d_{X_{F},\nu_{1}}(x)d_{(X_{I},X_{F}),\nu_{2}}(x,y)d_{X_{I},\nu_{3}}(y)}{\mathbb{E}_{(X_{I},X_{F}),\nu_{2}}(d_{X_{F},\nu_{1}}(X_{I})d_{X_{I},\nu_{3}}(X_{F}))}.$$
From the boundedness on the extreme runs and Proposition \ref{domNNRPinit}, there exists $K$ such that $\mathbb{E}_{(X_{I},X_{F}),\nu_{2}}(d_{X_{F},\nu_{1}}(X_{I})d_{X_{I},\nu_{3}}(X_{F}))\geq K$, and $\Vert d_{X_{F},\nu_{1}}\Vert,\Vert d_{X_{I},\nu_{3}}\Vert\leq K$.
Thus as $\nu_{2}$ becomes larger,
$$\Vert d_{x_{i_{1}},x_{i_{2}}}-d_{x_{i_{1}}}d_{x_{i_{2}}}\Vert_{\infty}\rightarrow 0,$$
independently of the shape of $\mu_{2}$. 
Finally by Lemma \ref{NNRPdes}, there exists $n$ such that if $i_{2}-i_{1}\geq n$,
$$d_{\pi}(\mu(\frac{\sigma_{\lambda}(i_{1})}{n},\frac{\sigma_{\lambda}(i_{r})}{n}),\mu(\frac{\sigma(i_{1})}{n})\otimes\mu(\frac{\sigma(i_{r})}{N}))\leq \epsilon.$$
\end{proof}
\section*{Acknowledgements}
I am grateful to Philippe Biane and Jean-Yves Thibon who suggested a problem that lead to this work. I wish also to thank  Roland Speicher and his team who welcomed me in Saarbruecken while I was doing this work. I first heard about the probabilistic approach to the descent statistic during the weekly workshop of the combinatoric team of Jean-Yves Thibon (during a talk of Matthieu Josua-Vergès) and my understanding of the subject greatly benefited from the reading of the paper \cite{bender2004asymptotics} of Bender, Helton and Richmond. 
\bibliographystyle{alpha}
\nocite{*}
\bibliography{DescentsetI}
\end{document}